\documentclass[a4paper,10pt,oneside,final]{article}
\usepackage[english]{babel}
\usepackage{fullpage}
\usepackage[T1]{fontenc}
\usepackage[affil-it]{authblk}
\usepackage[T1]{fontenc}
\usepackage{graphicx}
\usepackage{mathtools}
\usepackage{afterpage}
\usepackage{setspace}
\usepackage{blindtext}
\usepackage{etex}
\usepackage[utf8x]{inputenc}
\usepackage{amsmath,amssymb,amsfonts,amsthm}

\usepackage{hyperref}

\DeclareMathOperator{\Gal}{Gal}
\DeclareMathOperator{\rank}{Rank}
\DeclareMathOperator{\sing}{Sing}
\DeclareMathOperator{\Swan}{Swan}

\DeclareMathOperator{\geom}{geom}

\DeclareMathOperator{\Fr}{Fr}
\DeclareMathOperator{\FT}{FT}

\DeclareMathOperator{\Tr}{Tr}

\DeclareMathOperator{\Kl}{Kl}

\theoremstyle{plain}
\newtheorem{thm}{Theorem}
\newtheorem{cor}[thm]{Corollary}
\newtheorem{lem}[thm]{Lemma}

\newtheorem{prop}[thm]{Proposition}

\theoremstyle{definition}
\newtheorem{defn}{Definition}
\newtheorem{notat}{Notation}

\theoremstyle{remark}
\newtheorem{oss}{\textbf{Remark}}

\begin{document}

\title{A Polynomial Sieve and Sums of Deligne Type}
\author{Dante Bonolis%
  \thanks{correspondence address: \textit{dante.bonolis@math.ethz.ch}.}}
\date{}
\maketitle

\begin{abstract}
Let $f\in\mathbb{Z}[T]$ be any polynomial of degree $d>1$ and $F\in\mathbb{Z}[X_{0},...,X_{n}]$ an irreducible homogeneous polynomial of degree $e>1$ such that the projective hypersurface $V(F)$ is smooth. In this paper we present a new bound for $N(f,F,B):=|\{\textbf{x}\in\mathbb{Z}^{n+1}:\max_{0\leq i\leq n}|x_{i}|\leq B,\exists t\in\mathbb{Z}\text{ such that }f(t)=F(\textbf{x})\}|.$
To do this, we introduce a generalization of the power sieve (\cite{HB84}, \cite{Mun09}) and we extend two results by Deligne and Katz on estimates for additive and multiplicative characters in many variables.
\end{abstract}

\section{Introduction}
A fundamental role in Analytic Number Theory is played by the combination of sieve methods and bounds for sums involving algebraic functions as the additive and multiplicative characters. In this paper we present a blend of this type, introducing a generalization of the following results
\begin{itemize}
\item[$(i)$] The square sieve, or the more general power sieve of Heath-Brown and Munshi (\cite{HB84}, \cite{Mun09}),
\item[$(ii)$] Deligne's and Katz's estimates for additive and multiplicative characters in many variables (\cite{Del74}, \cite{Kat99} and \cite{Kat02a}),
\end{itemize}
and combining them to establish the following:
\begin{thm}
Let $f\in\mathbb{Z}[T]$ be any polynomial of degree $d>1$ and $F\in\mathbb{Z}[X_{0},...,X_{n}]$ an irreducible homogeneous polynomial of degree $e>1$ such that the projective hypersurface $V(F)$ is smooth. We denote
\[
N(f,F,B):=|\{\textbf{x}\in\mathbb{Z}^{n+1}:\max_{0\leq i\leq n}|x_{i}|\leq B,\exists t\in\mathbb{Z}\text{ such that }f(t)=F(\textbf{x})\}|.
\]
Then we have
\[
N(f,F,B)\ll_{d,e,n,\|f\|,\|F\|} B^{n+\frac{1}{n+2}}\log B^{\frac{n+1}{n+2}}.
\]
Where $\|f\|$, $\|F\|$ are the heights of $f$ and $F$ respectively, i.e. if $f=a_{n}T^{n}+...+a_{0}$, $\|f\|:=\max\{|a_{n}|,..., |a_{n}|\}$ and similarly for $\|F\|$.
\label{thm : box}
\end{thm}
\begin{oss}
For $f=T^{d}$ we obtain the same result as in \cite{Mun09}[Theorem $1.1$]. However, there is an issue in the proof of \cite{Mun09}[Theorem $1.1$]. Munshi suggested the work around that we implement in Lemma $\ref{lem : avesieve}$. Moreover, the bound in \cite{Mun09}[Theorem $1.1$] has been improved by Pierce and Heath-Brown for $n\geq 8$ (\cite{HP12}[Theorem $2$]).
\end{oss}

\begin{oss}
Using the large sieve and a result by Cohen (\cite{Coh81}), Serre has shown that (\cite[Theorem $2$, Chapter $13$]{Ser97}):
\[
N(f,F,B)\ll_{d,e,n,\|f\|,\|F\|} B^{n+\frac{1}{2}}.
\]
in particular, we improve Serre's bound as soon as $n\geq 2$. However, Serre's conjecture predicts $N(f,F,B)\ll B^{n}(\log B)^{m}$ for some $m$.
\end{oss}

As we already mentioned, Theorem $\ref{thm : box}$ is a combination of the polynomial sieve (which we do not state here since it is a bit technical) with bounds of the following type:
\begin{thm}
Let $p$ be a prime number and $F\in\mathbb{F}_{p}[X_{1},...,X_{n}]$ an irreducible homogeneous polynomial of degree $d\geq 1$ such that the projective hypersurface $V(F)\subset \mathbb{P}_{\mathbb{F}_{p}}^{n-1}$ is smooth. Then
\[
\sum_{\mathbf{x}\in\mathbb{F}_{p}^{n}}\Kl_{m}(F(\mathbf{x});p)\ll_{d,n}p^{\frac{n}{2}},
\]
where $\Kl_{m}(a;p)$ denotes the $m$-th Hyper-Kloosterman sum of parameter $(a;p)$:
\[
\Kl_{m}(a;p):=\frac{(-1)^{m-1}}{p^{(m-1)/2}}\sum_{\substack{y_{1},...,y_{m}\in\mathbb{F}_{p}^{\times}\\y_{1}\cdot...\cdot y_{m}=a}}e\Big(\frac{y_{1}+\cdots +y_{m}}{p}\Big).
\] 
\label{thm : esemp}
\end{thm}
\begin{oss}
Actually, we show a much more general version of Theorem $\ref{thm : esemp}$: first of all, we prove this result for general trace functions, not only for $\Kl_{m}$. Moreover $F$ will be a polynomial of Deligne type (see Definition $\ref{def : delpol}$) of which the homogeneous polynomials as in Theorem $\ref{thm : esemp}$ are particular cases.
\end{oss}
\begin{thm}
Let $p$ be a prime number and $m\geq 2$, and $F\in\mathbb{F}_{p}[X_{1},...,X_{n}]$ an irreducible homogeneous polynomial of degree $d>1$ such that the projective hypersurface $V(F)\subset \mathbb{P}_{\mathbb{F}_{p}}^{n-1}$ is smooth. For any $\mathbf{u}\in\mathbb{F}_{p}^{n}$ such that $V(\langle\mathbf{X},\mathbf{u}\rangle)$ is not tangent to $V(F)$ (i.e. $V(F)\cap V(\langle\mathbf{X},\mathbf{u}\rangle)$ is smooth of codimension $2$ in $\mathbb{P}_{\mathbb{F}_{p}}^{n-1}$) one has
\[
\sum_{\mathbf{x}\in\mathbb{F}_{p}^{n}}\Kl_{m}(F(\mathbf{x});p)e\Big(\frac{\langle\mathbf{x},\mathbf{u}\rangle}{p}\Big)\ll_{d,n}p^{\frac{n}{2}},
\]
where $e(z):=e^{2\pi iz}$. 
\label{thm : esem}
\end{thm}
\begin{oss}
Also in this case, we prove Theorem $\ref{thm : esem}$ not only for $\Kl_{m}$ but for a more general class of trace functions. Moreover we consider two irreducible polynomials $F,G\in\mathbb{F}_{p}[X_{1},...,X_{n}]$ of degree $d,e>1$ such that $V(F),V(G),V(F)\cap V(G)\subset \mathbb{P}_{\mathbb{F}_{p}}^{n-1}$ are smooth. On the other hand, the assumption of $F$ and $G$ homogeneous will be crucial for our proof. 
\end{oss}
\subsubsection{Organization of the paper}
In the next section we introduce the Polynomial sieve. In section $\ref{sec : poldel}$ we prove the bounds for sum of trace functions over polynomials of Deligne type. Finally, in the last part of this paper we prove Theorem $\ref{thm : box}$.
\newline\newline

\section{The polynomial sieve}
\subsubsection{The power sieve}
Let $\mathcal{A}:=(a(n))_{n\in\mathbb{N}}$ be a sequence of non-negative numbers such that $\sum_{n}a(n)<\infty$, and assume one knows the distribution of the sequence $\mathcal{A}$ for certain moduli. Then, one can ask what is the contribution of the $d$-th power in the above sum, i.e. what is the size of
\[
\mathcal{V}_{T^{d}}(\mathcal{A}):=\sum_{n}a(n^{d}).
\]
This question was investigated by Heath-Brown (\cite{HB84}), in the case of squares, and by Munshi (\cite{Mun09}) for general powers. Their arguments rely on two observations:
\begin{itemize}
\item[$(i)$] if a natural number $n$ is a $d$-th power then $n\mod p$ is a $d$-th power for any prime $p$,
\item[$(ii)$] if $p\equiv 1\mod d$, one decomposes the
characteristic function of $d$-th powers modulo $p$ over $\mathbb{F}_{p}^{\times}$ using multiplicative characters of order $d$
\begin{equation}
1_{T^{d},p}=\frac{1}{d}+\frac{1}{d}\sum_{\chi_{p}^{d}=1,\chi_{p}\neq 1}\chi_{p}.
\end{equation}
\end{itemize}
These two remarks lead to the following
\begin{lem}[\cite{HB84}, \cite{Mun09}]
Let $\mathcal{P}$ be a finite set of prime numbers $p\equiv 1\mod d$ and $P:=|\mathcal{P}|$. Assume that $(a(n))_{n\in\mathbb{N}}$ is a sequence of non-negative numbers such that $a(n)=0$ if $n=0$ or $n\geq e^{P}$. Then
\begin{equation}
\mathcal{V}_{T^{d}}(\mathcal{A})\ll_{d} P^{-1}\sum_{n}a(n)+P^{-2}\sum_{p\neq q\in\mathcal{P}}\sum_{\substack{\chi_{p}\neq 1, \chi_{p}^{d}=1 \\\chi_{q}\neq 1, \chi_{q}^{d}=1}}\Big|\sum_{n}a(n)\chi_{p}(n)\overline{\chi}_{q}(n)\Big|.
\label{eq : sheu}
\end{equation}
\label{lem : s}
\end{lem}

\subsubsection{A polynomial sieve}
Our goal is to generalize the power sieve to any polynomial with coefficients over $\mathbb{Z}$: let $h\in\mathbb{Z}[T]$ with $\deg h>1$, we denote $\mathcal{P}_{h}:=\{p\text { prime: } h(\mathbb{F}_{p})\neq\mathbb{F}_{p} \}$. We want to provide an upper bound for
\[
\mathcal{V}_{h}(\mathcal{A}):=\sum_{n\in h(\mathbb{Z})}a(n),
\]
where $(a(n))_{n\in\mathbb{N}}$ is as before. Also in this case, one has that if $n\in h(\mathbb{Z})$ then $n\in h(\mathbb{F}_{p})$ for all prime numbers $p$. Then, as in the case of the $d$-th powers, to generalize the power sieve we need a decomposition of the characteristic function $1_{h(\mathbb{F}_{p})}$ for any $p\in\mathcal{P}_{h}$. This is done in \cite[Proposition $6.7$]{FKM14a}, indeed Fouvry, Kowalski and Michel showed the following:
\begin{prop}
Let $p,\ell$ be two distinct primes and $h\in\mathbb{F}_{p}[T]$ a non trivial polynomial of degree $\deg h < p$. There exists an integer $k_{p}\geq 1$ and a finite number of trace functions $t_{i,p}$ associated to middle-extension $\ell$-adic sheaves $\mathcal{F}_{i,p}$, $1\leq i\leq k_{p}$ which are pointwise pure of weight $0$ and algebraic numbers $c_{i,p}\in\overline{\mathbb{Q}}$ such that
\begin{equation}
1_{h(\mathbb{F}_{p})}(x)=\sum_{i=1}^{k_{p}}c_{i,p}t_{i,p}(x)
\end{equation}
for all $x\in\mathbb{F}_{p}\setminus S_{h,p}$, where $S_{h,p}:=\{h(x):x\in\overline{\mathbb{F}}_{p}, h'(x)=0\}\cap\mathbb{F}_{p}$ and with the following properties:
\begin{itemize}
\item[$(i)$] the constants $k_{p}, |c_{i,p}|$ and the conductor of $\mathcal{F}_{i,p}$ (see Definition $\ref{defn : cond}$), $ c(\mathcal{F}_{i,p})$, are bounded only in terms of $\deg h$,
\item[$(ii)$] the sheaf $\mathcal{F}_{1,p}$ is trivial and none of $\mathcal{F}_{i,p}$ for $i\neq 1$ is geometrically trivial, and moreover
\[ 
c_{1,p}=\frac{|h(\mathbb{F}_{p})|}{p}+O_{\deg h}(p^{-\frac{1}{2}})
\]
\item[$(iii)$] all $\mathcal{F}_{i,p}$ are tame.
\end{itemize}
\label{thm : cite}
\end{prop} 
Using this we can prove
\begin{lem}
Let $h$ be a polynomial in $\mathbb{Z}[T]$ with $\deg h=d\geq 2$. Let $\mathcal{P}\subset\mathcal{P}_{h}$ be a finite subset of $\mathcal{P}_{h}$ and denote $P:=|\mathcal{P}|$. Then for any sequence $(a(n))_{n\in\mathbb{N}}$ of non-negative numbers such that $a(n)=0$ if $|\{p\in\mathcal{P}:n\in S_{h,p}\mod p)\}|\geq \frac{P}{2d}$ we have
\begin{equation}
\mathcal{V}_{h}(\mathcal{A})\ll_{d}P^{-1}\sum_{n}a(n)+ P^{-2}\sum_{p\neq q\in\mathcal{P}}\sum_{\substack{i,j\\i\neq 1, j\neq 1}}\Big|\sum_{n}a(n)t_{i,p}(n)\overline{t}_{j,q}(n)\Big|,
\label{eq : polsiv}
\end{equation}
where the functions $t_{i,p}(\cdot)$ and the sets $S_{h,p}$ are the ones in the decomposition of $1_{h(\mathbb{F}_{p})}$ in Proposition $\ref{thm : cite}$.
\label{lem : sieve}
\end{lem}
\begin{proof}
For any $p\in\mathcal{P}_{h}$ with $p>d$, an application of Proposition $\ref{thm : cite}$ leads to the decomposition
\[
1_{h(\mathbb{F}_{p})}=\frac{|h(\mathbb{F}_{p})|}{p}+\sum_{i=2}^{k_{p}}c_{i,p}t_{i,p}+O_{d}(p^{-\frac{1}{2}})
\]
over $\mathbb{F}_{p}\setminus S_{h,p}$, where $S_{h,p}=\{h(x)\in\overline{\mathbb{F}}_{p}:h'(x)=0\}\cap\mathbb{F}_{p}$ and with $k_{p},c_{i,p}\ll_{d}1$. Then we consider the weighted sum
\begin{equation}
\Sigma :=\sum_{n}a(n)\Big|\Big(\sum_{p\in\mathcal{P}}\sum_{i= 2}^{k_{p}}c_{i,p}t_{i,p}(n)\Big)\Big|^{2}.
\end{equation}
Thanks to the fact that $h(\mathbb{F}_{p})\neq \mathbb{F}_{p}$ (since $\mathcal{P}\subset\mathcal{P}_{h}$), one has that (\cite[Proposition $2.11$]{Tur95}) $|h(\mathbb{F}_{p})|\leq p-\frac{p-1}{d}$. Thus we have that for any $p\in\mathcal{P}$ and for any $z\in h(\mathbb{F}_{p})\setminus S_{h,p}$ the following inequality
\[
\sum_{i= 2}^{k_{p}}c_{i,p}t_{i,p}(z)=1-\frac{|h(\mathbb{F}_{p})|}{p}+O_{d}(p^{-\frac{1}{2}})\geq \frac{1}{d} +O_{d}(p^{-\frac{1}{2}}),
\]
holds. Hence if $n\in h(\mathbb{Z})$ and $a(n)\neq 0$ one has
\[
\sum_{p\in\mathcal{P}}\sum_{i= 2}^{k_{p}}c_{i,p}t_{i,p}(n)\geq \frac{P}{d} +O_{d}(P^{\frac{1}{2}})-\sum_{p\in\mathcal{P}:n\in S_{h,p}}1\geq \frac{P}{2d} +O_{d}(P^{\frac{1}{2}}),
\]
where the last step uses the vanishing assumption on the sequence $(a(n))_{n\in\mathbb{N}}$. Hence $P^{2}\mathcal{V}_{h}(\mathcal{A})\ll_{d} \Sigma$ by positivity. Opening the square in $\Sigma$ we get the result.
\end{proof}
\begin{oss}
Let us consider $h=T^{2}$ in Lemma $\ref{lem : sieve}$. For any prime number $p$ one has that $1_{T^{2},p}=\frac{1}{2}\Big(1+\Big(\frac{\cdot}{p}\Big)\Big)$. Thus, $(\ref{eq : polsiv})$ becomes
\[
\mathcal{V}_{T^{2}}(\mathcal{A})\ll P^{-1}\sum_{n}a(n)+ P^{-2}\sum_{p\neq q\in\mathcal{P}}\Big|\sum_{n}a(n)\Big(\frac{n}{p}\Big)\Big(\frac{n}{q}\Big)\Big|.
\]
Hence, we cover Heath-Brown's square sieve. Similarly, if we specialize Lemma $\ref{lem : sieve}$ to the case $h=T^{d}$ we obtain Munshi's power sieve.
\end{oss}

\begin{oss}
In \cite{Bro15}, Browing generalizes the Power Sieve introducing the weighted sum
\[
\Sigma' :=\sum_{n}a(n)\Big|\Big(\sum_{p\in\mathcal{P}}(\alpha+(\nu_{h,p}(n)-1)(d-\nu_{h,p}(n)))\Big|^{2},
\]
for a suitable $\alpha>0$ and where $\nu_{h,p}(n):=|\{x\in\mathbb{F}_{p}:h(x)=n\}|$. The advantage of his approach is that the quantities $\nu_{h,p}(n)$ are more concrete respect to the ones used in our paper. On the other hand, the paramater $\alpha$ depends on the polynomial $h$ and, a priori, it is not clear what should be the optimal choice. In our setting instead, all the quantities are determined by Proposition $\ref{thm : cite}$. Moreover, in our application we do not really need to compute the $t_{i}$s: it will be enough to know that these functions are tamely ramified everywhere.
\end{oss}
\begin{oss}
The vanishing condition on the sequence $(a(n))_{n\in\mathbb{N}}$ is necessary: let $h=T^{d}$ and $m$ be a number such that $p|m$ for any $p\in\mathcal{P}$, then $m=0$ or $m\geq e^{P}$. Consider the sequence
\[
a(n)=
\begin{cases}
1		&\text{if $n=m^{d}$}\\
0		&\text{if $n\neq m^{d}$}.
\end{cases}
\]
For this sequence $\mathcal{V}_{T^{d}}(\mathcal{A})=1$ while the right hand side in $(\ref{eq : polsiv})$ is $O(P^{-1})$.
\end{oss}

\section{Sums of Trace functions over polynomials of Deligne type}
\subsection*{Notation and statements of the general version of Theorem $\ref{thm : esemp}$ and Theorem $\ref{thm : esem}$}
In this section we recall some notion of the formalism of trace functions and state the general version of Theorem $\ref{thm : esemp}$ and Theorem $\ref{thm : esem}$. For a general introduction on this subject we refer to \cite{FKM14b}. Basic statements and references can also be founded in \cite{FKM15}. The main examples of trace functions we should have in mind are
\begin{itemize}
\item[$(i)$] For any $f\in\mathbb{F}_{p}[T]$, the function $x\mapsto e(f(x)/p)$: this is the trace function attached to the Artin-Schreier sheaf $\mathcal{L}_{e(f/p)}$.
\item[$(ii)$] For any $h\in\mathbb{F}_{p}[T]$, the trace functions $t_{i,p}$ appearing in Proposition $\ref{thm : cite}$. Notice that for $h=T^{d}$, the $t_{i,p}$s are just the multiplicative characters of order $d$.
\item[$(iii)$] The $n$-th Hyper-Kloosterman sums: the map
\[
x\mapsto\Kl_{n}(x;p):=\frac{(-1)^{n-1}}{p^{(n-1)/2}}\sum_{\substack{y_{1},...,y_{n}\in\mathbb{F}_{p}^{\times}\\y_{1}\cdot...\cdot y_{n}=x}}e\Big(\frac{y_{1}+\cdots +y_{n}}{p}\Big).
\]
can be seen as the trace function attached to the \textit{Kloosterman sheaf} $\mathcal{K}\ell_{n}$ (see \cite{Kat88} for the definition of such sheaf and for its basic properties).
\end{itemize}
\begin{defn}[\cite{FKM19}[pp. $4-6$]] Let $\mathcal{F}$ be a constructible $\ell$-adic sheaf on $\overline{\mathbb{P}}_{\mathbb{F}_{q}}^{1}$ and $j:U\hookrightarrow\overline{\mathbb{P}}_{\mathbb{F}_{q}}^{1}$ the largest dense open subset of $\overline{\mathbb{P}}_{\mathbb{F}_{q}}^{1}$ where $\mathcal{F}$ is lisse. The \textit{conductor of $\mathcal{F}$} is defined as
\[
c(\mathcal{F}):=\rank(\mathcal{F})+|\sing(\mathcal{F})|+\sum_{x}\text{Swan}_{x}(j_{*}j^{*}\mathcal{F})+\dim H_{c}^{0}(\overline{\mathbb{A}}_{\mathbb{F}_{q}},\mathcal{F}),
\]
where
\begin{enumerate}
\item[$i)$] $\rank(\mathcal{F}):=\dim \mathcal{F}_{x}$, for any $x$ where $\mathcal{F}$ is lisse.
\item[$ii)$] $\sing(\mathcal{F}):=\{x\in\overline{\mathbb{P}}_{\mathbb{F}_{q}}^{1}:\mathcal{F} \text{ is not lisse at }x\}$.
\item[$iii)$] For any $x\in\overline{\mathbb{P}}_{\mathbb{F}_{q}}^{1}$, $\Swan_{x}(j_{*}j^{*}\mathcal{F})$ is the \textit{Swan conductor of $\mathcal{F}$ at $x$} (see \cite{Kat80}[$4.4,4.5$] for the defnition and properties of the sheaf $j_{*}j^{*}\mathcal{F}$ and \cite{Kat88}[Chapter $1$] for the definition of the swan conductor).
\end{enumerate}
\label{defn : cond}
\end{defn}
\begin{oss}
For a a constructible $\ell$-adic sheaf $\mathcal{F}$ one has that (see \cite[$4.4,4.5$]{Kat80})
\[
\dim H_{c}^{0}(\overline{\mathbb{A}}_{\mathbb{F}_{q}}^{1},\mathcal{F})\leq\sum_{s\in \sing (\mathcal{F})}\dim (\mathcal{F}_{s}).
\]
\label{oss : zero}
\end{oss}
\begin{oss}
We recall that if $\mathcal{F}$ is middle-extension sheaf then
\[
c(\mathcal{F}):=\rank(\mathcal{F})+|\sing(\mathcal{F})|+\sum_{x}\text{Swan}_{x}(\mathcal{F}),
\]
since in this case $\mathcal{F}\cong j_{*}j^{*}\mathcal{F}$ and $\dim H_{c}^{0}(\overline{\mathbb{A}}_{\mathbb{F}_{q}^{1}},\mathcal{F})=0$.
\end{oss}
\begin{oss}
If $\mathcal{F}$ is an irreducible middle-extension sheaf on $\overline{\mathbb{P}}_{\mathbb{F}_{q}}^{1}$, such that $\mathcal{F}$ is lisse on $\overline{\mathbb{A}}_{\mathbb{F}_{q}}^{1}$ and tame at $\infty$ then $\mathcal{F}$ is geometrically trivial. Indeed, the Grothendieck–Ogg–Shafarevich Formula implies that
\begin{equation}
\dim (H_{c}^{0}(\overline{\mathbb{A}}_{\mathbb{F}_{q}}^{1},\mathcal{F}))-\dim (H_{c}^{1}(\overline{\mathbb{A}}_{\mathbb{F}_{q}}^{1},\mathcal{F}))+\dim (H_{c}^{2}(\overline{\mathbb{A}}_{\mathbb{F}_{q}}^{1},\mathcal{F}))=\rank (\mathcal{F})>0,
\label{eq : listam}
\end{equation}
On the other hand, $H_{c}^{0}(\overline{\mathbb{A}}_{\mathbb{F}_{q}}^{1},\mathcal{F})=0$ since $\mathcal{F}$ is a middle-extension $\ell$-adic sheaf. Moreover, if $\mathcal{F}$ is an irreducible $\ell$-adic sheaf then  $\dim (H_{c}^{2}(\overline{\mathbb{A}}_{\mathbb{F}_{q}}^{1}\mathcal{F}))=1$ if and only if $\mathcal{F}$ is geometrically trivial and it is $0$ otherwise. Combining this with $(\ref{eq : listam})$, we get that $\mathcal{F}$ is geometrically trivial.
\label{oss : listam}
\end{oss}

\begin{defn}
Let $\mathcal{F}$ be a Fourier sheaf and $\psi$ a non trivial character over $\mathbb{F}_{q}$. Fix an integer $e\geq 1$ and consider the morphism
\[
\overline{\mathbb{A}}_{\mathbb{F}_{q}}^{1}\rightarrow\overline{\mathbb{A}}_{\mathbb{F}_{q}}^{1},\quad x\mapsto x^{e}
\]
Then we can define the $\ell$-adic sheaf $T_{e}(\mathcal{F}):=\FT_{\psi}\Big([x\mapsto x^{e}]_{*}\mathcal{F}\Big)$.
\label{defn : power}
\end{defn}
Notice that
\[
\begin{split}
t_{T_{e}(\mathcal{F})}(y)&=-\frac{1}{q^{1/2}}\sum_{x\in\mathbb{F}_{q}}\psi (xy)t_{[x\mapsto x^{e}]_{*}\mathcal{F}}(x)\\&=-\frac{1}{q^{1/2}}\sum_{x\in\mathbb{F}_{q}}\psi (xy)\sum_{z^{e}=x}t_{\mathcal{F}}(z)\\&=-\frac{1}{q^{1/2}}\sum_{z\in\mathbb{F}_{q}}\psi (z^{e}y)t_{\mathcal{F}}(z).
\end{split}
\]
Then we have
\begin{lem}
Let $\mathcal{F}$ be a Fourier sheaf and $\psi$ a non trivial character over $\mathbb{F}_{q}$. If $e<p$, then
\[
c(T_{e}(\mathcal{F}))\ll_{e} c(\mathcal{F})^{4}.
\]

\end{lem}
\begin{proof}
Since $T_{e}(\mathcal{F}):=\FT_{\psi}\Big([x\mapsto x^{e}]_{*}\mathcal{F}\Big)$, we have that $c(T_{e}(\mathcal{F}))\leq 10 c([x\mapsto x^{e}]_{*}\mathcal{F})^{2}$ thanks to \cite{FKM15}[Proposition $8.2$]. Hence, we need to bound $c([x\mapsto x^{e}]_{*}\mathcal{F})$. First, observe that $\rank ([x\mapsto x^{e}]_{*}\mathcal{F})=e\rank\mathcal{F}$ and that $\sing ([x\mapsto x^{e}]_{*}\mathcal{F})\subset\sing\mathcal{F}\cup [x\mapsto x^{e}]_{*}(\sing\mathcal{F})$. Let us bound the Swan conductors. Applying \cite{Kat02b}[Lemma $1.6.4.1$] twice, for $\mathcal{F}$ and $\overline{\mathbb{Q}}_{\ell}$, we get
\[
\Swan_{x}([x\mapsto x^{e}]_{*}\mathcal{F})=\rank(\mathcal{F})\Swan_{x}([x\mapsto x^{e}]_{*}\overline{\mathbb{Q}}_{\ell})+\sum_{t:t^{e}=x}\Swan_{x}(\mathcal{F}).
\]
On the other hand, $[x\mapsto x^{e}]_{*}\overline{\mathbb{Q}}_{\ell}$ is tamely ramified everywhere since $e<p$. Thus, $\Swan_{x}([x\mapsto x^{e}]_{*}\mathcal{F})\leq e\Swan_{x}(\mathcal{F})$. Putting all together, we conclude that $c([x\mapsto x^{e}]_{*}\mathcal{F})\leq 5e^{2}c(\mathcal{F})^{2}$ and this proves the Lemma.
\end{proof}

\begin{defn}
Let $f\in\mathbb{F}_{q}[X_{1},...,X_{n}]$ be a polynomial in $n$ variables of degree $d\geq 1$, say:
\[
f=f_{d}+...+f_{0},
\]
where $f_{i}$ are homogeneous polynomials of degree $i$. We say that $f$ is of \textit{Deligne type} if $p\nmid d$ and the zero set, $V(f_{d})$, of $f_{d}$ defines a smooth hypersurface in $\overline{\mathbb{P}}_{\mathbb{F}_{q}}^{n-1}$.
\label{def : delpol}
\end{defn}
At this point it is useful to prove the following result on polynomial of Deligne type
\begin{prop}
Let $X\subset \overline{\mathbb{P}}_{\mathbb{F}_{q}}^{n}$ be a smooth variety of complete intersection of dimension $r$ and multidegree $(d_{1},...,d_{r})$ and assume that $X\cap V(X_{0})$ is smooth and of codimension $1$ in $X$. Then
\[
|(X\setminus (X\cap V(X_{0}))(\mathbb{F}_{q})|=q^{r}+O_{d_{1},...d_{r},r}(q^{\frac{r}{2}})
\]
In particular if $f\in\mathbb{F}_{q}[X_{1},...,X_{n}]$ is a polynomial of Deligne type such that the homogenization, $F\in\mathbb{F}_{q}[X_{0},X_{1},...,X_{n}]$, of $f$ define a smooth projective hypersurface $V(F)\subset\overline{\mathbb{P}}_{\mathbb{F}_{q}}^{n}$, then
\[
|V(f)(\mathbb{F}_{q})|=q^{n-1}+O_{d,n}(q^{\frac{n-1}{2}}),
\]
where $V(f)\subset\overline{\mathbb{A}}_{\mathbb{F}_{q}}^{n}$ is the affine hypersurface attached to $f$.
\label{prop : coundel}
\end{prop}
\begin{proof}
Combining \cite{Del74}[Theorem $8.1$] with \cite{Bom78}[Theorem $1$A], one has that
\[
|X(\mathbb{F}_{q})|=\sum_{\substack{j=r+1\\ j\text{ even}}}^{2r}q^{\frac{j}{2}}+O_{d_{1},...d_{r},r}(q^{\frac{r}{2}}).
\]
On the other hand, $X\cap V(X_{0})$ is a smooth variety of complete intersection of dimension $r-1$ and multidegree $(d_{1},...,d_{r},1)$ by hypothesis. Thus,
\[
|(X\cap V(X_{0}))(\mathbb{F}_{q})|=\sum_{\substack{j=r\\ j\text{ even}}}^{2r-2}q^{\frac{j}{2}}+O_{d_{1},...d_{r},r}(q^{\frac{r-1}{2}}).
\]
Then the result follows since $|(X\setminus (X\cap V(X_{0}))(\mathbb{F}_{q})|=|X(\mathbb{F}_{q})|-|(X\cap V(X_{0}))(\mathbb{F}_{q})|$.
\end{proof}

\begin{notat}
We use the following conventions:
\begin{itemize}
\item[$(i)$]if $Y$ is a scheme over a field $k$ and $\mathcal{F}$ is a constructible $\ell$-adic sheaf on $Y$ we will denote
\[
H_{c}^{i}(\overline{Y},\mathcal{F}):=H_{c}^{i}(Y\times\overline{k},\mathcal{F}).
\]
Moreover, we will denote by $t_{\mathcal{F}}$ the trace function attached to $\mathcal{F}$.
\item[$(ii)$] Any scheme is a scheme over $\mathbb{F}_{q}$ and any morphism is a $\mathbb{F}_{q}$-morphism.
\item[$(iii)$] For any $g\in\mathbb{F}_{q}[X_{1},...,X_{n}]$ we denote by $G\in\mathbb{F}_{q}[X_{0},X_{1},...,X_{n}]$ its homogenization. Moreover, we denote the affine variety associated to $g$ by $V(g)$ and the projective one associated to $G$ by $V(G)$.
\end{itemize}
\label{notat : conv}
\end{notat}
Now we can state our main results
\begin{thm}
Fix $d,e\geq 1$, $p$ a prime number such that $p\nmid de$ and $q=p^{\alpha}$ for some $\alpha\geq 1$.  Let $\ell\neq p$ be a prime number and let $\mathcal{F}$ be a middle-extension $\ell$-adic sheaf on $\overline{\mathbb{A}}_{\mathbb{F}_{q}}^{1}$ of weight $0$. We assume that the geometrically irreducible components of $\mathcal{F}$ are not geometrically trivial, i.e. $H_{c}^{2}(\overline{\mathbb{A}}_{\mathbb{F}_{q}}^{1},\mathcal{F})=0$. Let $t_{\mathcal{F}}$ be the trace function of $\mathcal{F}$. Let $U$ be the maximal dense open subset of $\overline{\mathbb{A}}_{\mathbb{F}_{q}}^{1}$ where $\mathcal{F}$ is lisse. Let $f,g\in\mathbb{F}_{q}[X_{1},...X_{n+1}]$ be polynomials of Deligne type in $n+1$ variables respectively of degree $d$ and $e$. Assume that $V(G)$ is smooth projective variety, $V(G)\cap V(F)$ is smooth of codimension $1$ in $V(G)$, and that $V(G)\cap V(F)\cap V(X_{0})$ is a smooth projective variety of codimension $2$ in $V(G)$. Let us consider the morphism $f: V(g)\rightarrow\overline{\mathbb{A}}_{\mathbb{F}_{q}}^{1}$. If any geometrically irreducible component of $\mathcal{F}$, $\mathcal{F}_{i}$, satisfies one of the following conditions:
\begin{itemize}
\item[$(i)$] there exists $s_{i}\in \sing (\mathcal{F}_{i})$ such that $f^{-1}(s_{i})$ is smooth;
\item[$(ii)$] the sheaf $\mathcal{F}_{i}$ is wildly ramified at $\infty$,
\end{itemize}
then we have
\begin{equation}
\sum_{\substack{\mathbf{x}\in \mathbb{F}_{q}^{n+1}\\g(\mathbf{x})=0}}t_{\mathcal{F}}(f(\mathbf{x}))= q^{n-1}\sum_{x\in\mathbb{F}_{q}}t_{\mathcal{F}}(x)+O_{d,e,n,c(\mathcal{F})}(q^{\frac{n}{2}}).
\end{equation}
\label{thm : 1}
\end{thm}

\begin{thm}
Fix $d,e\geq 1$, $p$ is a prime number such that $p\nmid de$ and $q=p^{\alpha}$ for some $\alpha\geq 1$. Let $F,G\in\mathbb{F}_{q}[X_{1},...X_{n}]$ be homogeneous polynomials in $n$ variables of degree $d$ and $e$ respectively. Suppose $V(G),V(F)\subset\overline{\mathbb{P}}_{\mathbb{F}_{p}}^{n-1}$ are smooth and that $V(F)\cap V(G)$ is smooth and of codimension $2$ in $\overline{\mathbb{P}}_{\mathbb{F}_{p}}^{n-1}$. Moreover assume that $T_{e}([x\mapsto x^{d}]^{*}\mathcal{F})$ is geometrically irreducible and not geometrically trivial. Then one has
\begin{equation}
\Big|\sum_{\mathbf{x}\in\mathbb{F}_{q}^{n}}t_{\mathcal{F}}(F(\mathbf{x}))\psi(G(\mathbf{x}))\Big|\ll_{d,e,n,c(\mathcal{F})}q^{\frac{n}{2}},
\label{eq : mix}
\end{equation}
where $\psi$ is a non-trivial additive character of $\mathbb{F}_{q}$.
\label{cor : mix}
\end{thm}

\subsection*{Remarks and related works}
\begin{itemize}
\item[$(i)$] It is possible to prove a weaker versions of Theorem $\ref{thm : 1}$ and Theorem $\ref{cor : mix}$ using Proposition $\ref{prop : coundel}$. Indeed, one starts writing 
\[
\sum_{\mathbf{x}\in\mathbb{F}_{q}^{n}}t_{\mathcal{F}}(f(\mathbf{x}))=\sum_{a\in\mathbb{F}_{q}}t_{\mathcal{F}}(a)\sum_{\mathbf{x}\in\mathbb{F}_{q}^{n}:f(\mathbf{x})=a}1.
\]
On the other hand, Proposition $\ref{prop : coundel}$ implies that $\sum_{\mathbf{x}\in\mathbb{F}_{q}^{n}:f(\mathbf{x})=a}1=q^{n-1}+E(f,a)$ where $E(f,a)=O_{d,n}(q^{\frac{n-1}{2}})$. Thus,
\begin{equation}
\begin{split}
\sum_{\mathbf{x}\in\mathbb{F}_{q}^{n}}t_{\mathcal{F}}(f(\mathbf{x}))&=q^{n-1}\sum_{a\in\mathbb{F}_{q}}t_{\mathcal{F}}(a)+\sum_{a\in\mathbb{F}_{q}}t_{\mathcal{F}}(a)E(f,a)\\&=q^{n-1}\sum_{a\in\mathbb{F}_{q}}t_{\mathcal{F}}(a)+O_{d,n,c(\mathcal{F})}(q^{\frac{n+1}{2}}),
\end{split}
\label{eq : trivial1}
\end{equation}
Roughly speaking, in Theorem $\ref{thm : 1}$ we are going to bound more carefully the sum $\sum_{a\in\mathbb{F}_{q}}t_{\mathcal{F}}(a)E(f,a)$ gaining a factor $q^{\frac{1}{2}}$. Similarly, one can prove trivially that
\begin{equation}
\Big|\sum_{\mathbf{x}\in\mathbb{F}_{q}^{n}}t_{\mathcal{F}}(F(\mathbf{x}))\psi(G(\mathbf{x}))\Big|\ll_{d,e,n,c(\mathcal{F})}q^{\frac{n}{2}+1}.
\label{eq : trivial2}
\end{equation}
Hence, in this case we improve the error term by a factor $q$.
\item[$(ii)$] Theorem $\ref{thm : 1}$ was already established in the case of $t_{\mathcal{F}}=\psi$ a non-trivial additive character (\cite{Del74}) and in the case of $t_{\mathcal{F}}=\chi$ a non-trivial multiplicative character (\cite{Kat99}).
\item[$(iii)$] A more general version of Theorem $\ref{cor : mix}$ was known already in the case $t_{\mathcal{F}}=\chi$ a non-trivial multiplicative character (\cite{Kat02a}). 
\end{itemize}

\section*{Proofs of Theorem $\ref{thm : 1}$ and Theorem \ref{cor : mix}}
\label{sec : poldel}
\subsubsection*{The Incidence variety}
Let $g(X_{1},...,X_{n+1})$ be as in Theorem $\ref{thm : 1}$ and $G(X_{0},...,X_{n+1})$ its homogenization. Following the notation of Katz in \cite{Kat02a}, we denote $X=V(G)\subset\overline{\mathbb{P}}_{\mathbb{F}_{q}}^{n+1}$. Similarly, let $F(X_{0},...,X_{n+1})$ be the homogenization of $f(X_{1},...,X_{n+1})$ and denote by $H=V(F)\subset\overline{\mathbb{P}}_{\mathbb{F}_{q}}^{n+1}$. Moreover, we denote $L=V(X_{0})$, $Z=X\cap H\cap L$ and $V=X\setminus (X\cap L)=V(g)\subset \overline{\mathbb{A}}_{\mathbb{F}_{q}}^{n+1}$. Let us consider
\[
f=H/L^d:V\longrightarrow\overline{\mathbb{A}}_{\mathbb{F}_{q}}^{1},
\]
and the constructible $\mathbb{Q}_{\ell}$-sheaf $f^{*}\mathcal{F}$ on $V$. By the Lefschetz trace formula, we get
\begin{equation}
\sum_{\mathbf{x}\in V(\mathbb{F}_{q})}t_{\mathcal{F}}(f(\mathbf{x}))=\sum_{\mathbf{x}\in V(\mathbb{F}_{q})}t_{f^{*}\mathcal{F}}(\mathbf{x})=\sum_{i}(-1)^{i}\Tr (\Fr_{\mathbb{F}_{q}}|H_{c}^{i}(\overline{V},f^{*}\mathcal{F})).
\label{eq : lef}
\end{equation}
To compute the right hand side of the equation $(\ref{eq : lef})$ we are going to use the same strategy of \cite[Theorem $8.4$]{Del74}, \cite{Kat99} and \cite{Kat02a}. One starts introducing the incidence variety
\[
\tilde{X}=\{(\mathbf{x},\lambda)\in X\times\overline{\mathbb{A}}_{\mathbb{F}_{q}}^{1}:F(\mathbf{x})-\lambda x_{0}^d=0\}.
\]
Then one has that
\begin{lem}
The incident variety, $\tilde{X}$,  is a smooth variety for any polynomial of Deligne type $f$. Moreover
\begin{enumerate}
\item  The second projection of $X\times\overline{\mathbb{A}}_{\mathbb{F}_{q}}^{1}$ induces a proper flat morphism
\[
\tilde{f}:\tilde{X}\longrightarrow\overline{\mathbb{A}}_{\mathbb{F}_{q}}^{1}
\]
with $\tilde{X}_{\lambda}:=\tilde{f}^{-1}(\lambda)=V(G)\cap V(F-\lambda X_{0}^{d})$ for any $\lambda\in\overline{\mathbb{A}}_{\mathbb{F}_{q}}^{1}$.
\item The affine variety $V$ can be viewed as an open subset of $\tilde{X}$ and one has that $\tilde{f}_{|V}=f$. Moreover, the closed complement of $V$ in $\tilde{X}$ is given by:
\[
\tilde{Z}:=(V(G) \cap V(F)\cap V(X_{0}))\times\overline{\mathbb{A}}_{\mathbb{F}_{q}}^{1}=Z\times \overline{\mathbb{A}}_{\mathbb{F}_{q}}^{1}.
\] 
\end{enumerate}
\label{lem : encaps}
\end{lem}
\begin{proof}
This can be found in \cite{Kat80}[Pages $173$-$174$]. 
\end{proof}
Arguing as \cite[Lemma $11$]{Kat99}, one rewrites the sum of the trace of the sheaf $\tilde{f}^{*}\mathcal{F}$ on $\tilde{X}$ as
\[
\sum_{x\in\tilde{X}(\mathbb{F}_{q})}t_{\tilde{f}^{*}\mathcal{F}}(x)=\sum_{i,j}(-1)^{i+j}\Tr(\Fr |H_{c}^{j}(\overline{\mathbb{A}}_{\mathbb{F}_{q}}^{1},\mathcal{F}\otimes R^{i}\tilde{f}_{*}\overline{\mathbb{Q}}_{\ell})).
\]
Hence in order to compute the right hand side of the above equality we need to understand the size of the cohomology groups $H_{c}^{j}(\overline{\mathbb{A}}_{\mathbb{F}_{q}}^{1},\mathcal{F}\otimes R^{i}\tilde{f}_{*}\overline{\mathbb{Q}}_{\ell})$ and the action of the Frobenius automorphism on these groups.

\subsubsection*{Properties of $R^{i}\tilde{f}_{*}\overline{\mathbb{Q}}_{\ell}$}
We can state the following proposition concerning the property of the higher direct image sheaves $R^{i}\tilde{f}_{*}\overline{\mathbb{Q}}_{\ell}$ on $\overline{\mathbb{A}}_{\mathbb{F}_{q}}^{1}$ (recall that $\tilde{f}$ is proper).
\begin{prop}
Consider the morphism
\[
\tilde{f}:\tilde{X}\longrightarrow\overline{\mathbb{A}}_{\mathbb{F}_{q}}^{1}
\]
then the following properties hold:
\begin{itemize}
\item[$(i)$] $\tilde{f}$ is smooth of relative dimension $n-1$ in a Zariski open neighborhood of the origin $0$ in $\overline{\mathbb{A}}_{\mathbb{F}_{q}}^{1}$. Morover, the sheaves $R^{i}\tilde{f}_{*}\overline{\mathbb{Q}}_{\ell}$ are lisse in a Zariski open neighborhood of the origin.
\item[$(ii)$] If $\lambda\in\overline{\mathbb{A}}_{\mathbb{F}_{q}}^{1}$ is such that $\tilde{X}_{\lambda}$ is singular, then $\tilde{X}_{\lambda}$ has only isolated singularities. 
\item[$(iii)$] for any $i\geq n+1$ the sheaf $R^{i}\tilde{f}_{*}\overline{\mathbb{Q}}_{\ell}$ is lisse on $\overline{\mathbb{A}}_{\mathbb{F}_{q}}^{1}$ and pure of weight $i$,
\item[$(iv)$] the sheaves $R^{i}\tilde{f}_{*}\overline{\mathbb{Q}}_{\ell}$ are tamely ramified at $\infty$ for all $i$.
\item[$(v)$] for $i=n$, denote $j:U\longrightarrow\overline{\mathbb{A}}_{\mathbb{F}_{q}}^{1}$ the inclusion of an open dense subset $U$  of $\overline{\mathbb{A}}_{\mathbb{F}_{q}}^{1}$ where $R^{n}\tilde{f}_{*}\overline{\mathbb{Q}}_{\ell}$ is lisse, then one has the exact sequence:
\begin{equation}
0\longrightarrow\mathcal{P}\longrightarrow R^{n}\tilde{f}_{*}\overline{\mathbb{Q}}_{\ell}\longrightarrow j_{*}j^{*} R^{n}\tilde{f}_{*}\overline{\mathbb{Q}}_{\ell}\longrightarrow 0,
\label{eq : exact}
\end{equation}
where $j_{*}j^{*}R^{n}\tilde{f}_{*}\overline{\mathbb{Q}}_{\ell}$ is geometrically constant and $\mathcal{P}$ punctual.
\end{itemize}
\label{prop : 1}
\end{prop}
\begin{proof}
The fiber $\tilde{X}_{0}=V(G)\cap V(F)$ is smooth by hypothesis, then $\tilde{f}$ is smooth of relative dimension $n-1$ in a Zariski open neighborhood $0$. Moreover the sheaves $R^{i}\tilde{f}_{*}\overline{\mathbb{Q}}_{\ell}$ are lisse in a Zariski open neighborhood of the origin $0$ in $\overline{\mathbb{A}}_{\mathbb{F}_{q}}^{1}$ thanks to \cite{SGA4}[Expos\'e XV, Theorem $2.1$, page $192$ and Expos\'e XVI, Theorem $2.1$, page $213$] and this proves $(i)$. Then: $(ii)$ is \cite{Kat02a}[Proposition $7.1$], $(iii)$, $(v)$ are \cite{Kat02a}[Proposition $7.2$], $(iv)$ is \cite{Kat02a}[Proposition $7.3$].
\end{proof}
\begin{cor}
For any $i\geq n+1$, the sheaf $R^{i}\tilde{f}_{*}\overline{\mathbb{Q}}_{\ell}$ is geometrically constant on $\overline{\mathbb{A}}_{\mathbb{F}_{q}}^{1}$.
\label{cor : geomconst}
\end{cor}
\begin{proof}
For any $i\geq n+1$, the sheaf $R^{i}\tilde{f}_{*}\overline{\mathbb{Q}}_{\ell}$ is lisse on $\overline{\mathbb{A}}_{\mathbb{F}_{q}}^{1}$ and tame at $\infty$, thus it is geometrically constant on $\overline{\mathbb{A}}_{\mathbb{F}_{q}}^{1}$ thanks to Remark $\ref{oss : listam}$.
\end{proof}
We conclude this section by calculating the cohomology of the sheaf $R^{n}\tilde{f}_{*}\overline{\mathbb{Q}}_{\ell}$:
\begin{lem}
One has that:
\[
H_{c}^{i}(\overline{\mathbb{A}}_{\mathbb{F}_{q}}^{1},R^{n}\tilde{f}_{*}\overline{\mathbb{Q}}_{\ell})=
\begin{cases}
H_{c}^{0}(\overline{\mathbb{A}}_{\mathbb{F}_{q}}^{1},\mathcal{P}),       & \text{if $i = 0$,} \\
0, & \text{if $ i=1$,} \\
H_{c}^{2}(\overline{\mathbb{A}}_{\mathbb{F}_{q}}^{1},j_{*}j^{*} R^{n}\tilde{f}_{*}\overline{\mathbb{Q}}_{\ell})  & \text{if $i=2$.} \\
\end{cases}
\]
\label{lem : 4}
\end{lem}
\begin{proof}
Applying \cite{Kat80}[section $4.5.2$] one has that
\[
H_{c}^{i}(\overline{\mathbb{A}}_{\mathbb{F}_{q}}^{1},R^{n}\tilde{f}_{*}\overline{\mathbb{Q}}_{\ell})=
\begin{cases}
H_{c}^{0}(\overline{\mathbb{A}}_{\mathbb{F}_{q}}^{1},\mathcal{P}),       & \text{if $i = 0$,} \\
H_{c}^{i}(\overline{\mathbb{A}}_{\mathbb{F}_{q}}^{1},j_{*}j^{*} R^{n}\tilde{f}_{*}\overline{\mathbb{Q}}_{\ell})  & \text{if $i=1,2$.} \\
\end{cases}
\]
and that $H_{c}^{1}(\overline{\mathbb{A}}_{\mathbb{F}_{q}}^{1},\mathcal{P})=H_{c}^{2}(\overline{\mathbb{A}}_{\mathbb{F}_{q}}^{1},\mathcal{P})=H_{c}^{0}(\overline{\mathbb{A}}_{\mathbb{F}_{q}}^{1},j_{*}j^{*} R^{n}\tilde{f}_{*}\overline{\mathbb{Q}}_{\ell})=0$. Using the Grothendieck–Ogg–Shafarevich Formula for the sheaf $j_{*}j^{*} R^{n}\tilde{f}_{*}\overline{\mathbb{Q}}_{\ell}$ (which is geometrically constant) one obtains
\[
\dim (H_{c}^{2}(\overline{\mathbb{A}}_{\mathbb{F}_{q}}^{1},j_{*}j^{*} R^{n}\tilde{f}_{*}\overline{\mathbb{Q}}_{\ell}))-\dim (H_{c}^{1}(\overline{\mathbb{A}}_{\mathbb{F}_{q}}^{1},j_{*}j^{*} R^{n}\tilde{f}_{*}\overline{\mathbb{Q}}_{\ell}))= \rank (j_{*}j^{*} R^{n}\tilde{f}_{*}\overline{\mathbb{Q}}_{\ell}).
\]
On the other hand, $\dim (H_{c}^{2}(\overline{\mathbb{A}}_{\mathbb{F}_{q}}^{1},j_{*}j^{*} R^{n}\tilde{f}_{*}\overline{\mathbb{Q}}_{\ell}))\leq \rank (j_{*}j^{*} R^{n}\tilde{f}_{*}\overline{\mathbb{Q}}_{\ell})$ and then the lemma follows.
\end{proof}

\subsubsection*{Bounds for the conductor of the cohomology groups}
Before to go on, it is useful to recall the hypothesis on the sheaf $\mathcal{F}$: any irreducible component of $\mathcal{F}$ is not geometrically trivial, i.e. $H_{c}^{2}(\overline{\mathbb{A}}_{\mathbb{F}_{q}}^{1},\mathcal{F})=0$ and satisfies one of the following conditions:
\begin{itemize}
\item[$(i)$] there exists $s_{i}\in \sing (\mathcal{F}_{i})$ such that $f^{-1}(s_{i})$ is smooth;
\item[$(ii)$] the sheaf $\mathcal{F}_{i}$ is wildly ramified at $\infty$.
\end{itemize}
Then we prove the following
\begin{lem} If $j=2$ or $i\geq 2n-1$ then $H_{c}^{j}(\overline{\mathbb{A}}_{\mathbb{F}_{q}}^{1},\mathcal{F}\otimes R^{i}\tilde{f}_{*}\overline{\mathbb{Q}}_{\ell})=0$.
\label{lem : cohom2}
\end{lem}
\begin{proof}
We have two cases
\begin{itemize}
\item[$(i)$]for $j=2$ and any $i$  we can argue as follow: we can consider an open Zariski dense set $\overline{U}_{i}$ where $\mathcal{F}$ and $R^{i}\tilde{f}_{*}\overline{\mathbb{Q}}_{\ell}$ are both lisse and pointwise pure. Observe that $H_{c}^{2}(\overline{U}_{i},\mathcal{F}\otimes R^{i}\tilde{f}_{*}\overline{\mathbb{Q}}_{\ell})$ does not vanish if and only if there exists a geometrically irreducible component in the Jordan-Holder decomposition of $\mathcal{F}$, $\mathcal{F}_{i}$, and a geometrically irreducible component in the Jordan-Holder decomposition of $R^{i}\tilde{f}_{*}\overline{\mathbb{Q}}_{\ell}$, $\mathcal{G}$, such that $\mathcal{F}_{i}\cong_{\geom}\mathcal{G}$. Now this is not the case: indeed if $\mathcal{F}_{i}$ is wildly ramified at $\infty$ then a contradiction arises since $\mathcal{G}$ is tame at $\infty$ because the sheaves $R^{i}\tilde{f}_{*}\overline{\mathbb{Q}}_{\ell}$ is tame at $\infty $ (Proposition $\ref{prop : 1}$, $(iii)$). If there exists a $s_{i}\in \sing (\mathcal{F}_{i})$ such that $\tilde{X}_{s}$ is smooth, then by part $(i)$ of Proposition $\ref{prop : 1}$, $R^{i}\tilde{f}_{*}\overline{\mathbb{Q}}_{\ell}$ is lisse at $s_{i}$ and thus $\mathcal{G}$ is lisse at $s_{i}$.  Then $\mathcal{F}_{i}$ and $\mathcal{G}$ are not geometrically isomorphic because $s_{i}\in\sing (\mathcal{F}_{i})$ and $s_{i}\notin\sing(\mathcal{G})$.
\item[$(ii)$] if $i\geq 2n-1$ then a consequence of the proper base change Theorem tells us that for any $\lambda\in \overline{\mathbb{A}}_{\mathbb{F}_{q}}^{1}$
\[
(R^{i}\tilde{f}_{*}\overline{\mathbb{Q}}_{\ell})_{\lambda }  \cong H_{c}^{i}(\tilde{X}_{\lambda},\overline{\mathbb{Q}}_{\ell}).
\]
On the other hand, the right hand side of the equation above vanishes since $\tilde{f}$ has dimension $n-1$.
\end{itemize}
\end{proof}

\begin{prop}
For any $i\geq 0$ one has that $c(\mathcal{F}\otimes R^{i}\tilde{f}_{*}\overline{\mathbb{Q}}_{\ell})\ll_{d,e,n,c(\mathcal{F})}1.$
\label{lem : 2}
\end{prop}
\begin{proof}
Thanks to \cite[Proposition $8.2$, part $(3)$]{FKM15}, we know that $c(\mathcal{F}\otimes R^{i}\tilde{f}_{*}\overline{\mathbb{Q}}_{\ell})\leq 5c(\mathcal{F})^{2}c(R^{i}\tilde{f}_{*}\overline{\mathbb{Q}}_{\ell})^{2}$. Hence, we need to prove that $
c(R^{i}\tilde{f}_{*}\overline{\mathbb{Q}}_{\ell})\ll_{d,e,n}1
$. At this point it is useful to recall that the conductor of the constructible $\ell$-adic sheaf $ R^{i}\tilde{f}_{*}\overline{\mathbb{Q}}_{\ell}$ is defined as
\begin{equation}
\begin{split}
c(R^{i}\tilde{f}_{*}\overline{\mathbb{Q}}_{\ell})&=\rank(R^{i}\tilde{f}_{*}\overline{\mathbb{Q}}_{\ell})+\sing(R^{i}\tilde{f}_{*}\overline{\mathbb{Q}}_{\ell})\\&+\sum_{x}\text{Swan}_{x}(j_{*}j^{*}R^{i}\tilde{f}_{*}\overline{\mathbb{Q}}_{\ell})+\dim H_{c}^{0}(\overline{\mathbb{A}}_{\mathbb{F}_{q}},R^{i}\tilde{f}_{*}\overline{\mathbb{Q}}_{\ell}),
\label{eq : condprop}
\end{split}
\end{equation}
where $j:U\hookrightarrow\overline{\mathbb{P}}_{\mathbb{F}_{q}}^{1}$ is the largest dense open subset of $\overline{\mathbb{P}}_{\mathbb{F}_{q}}^{1}$ where $R^{i}\tilde{f}_{*}\overline{\mathbb{Q}}_{\ell}$ is lisse. We will handle each term in the right hand side of $(\ref{eq : condprop})$ separately.
\begin{lem}
For any $i\geq 0$ we have $\sing(R^{i}\tilde{f}_{*}\overline{\mathbb{Q}}_{\ell})\ll_{d,e,n} 1$.
\label{lem : sing}
\end{lem}
\begin{proof}
Let $S:=\{\lambda\in\overline{\mathbb{A}}_{\mathbb{F}_{q}}^{1}:\tilde{f}\text{ is not smooth in }\lambda\}$. Thanks to Proposition $\ref{prop : 1}.(ii)$, we have $|\sing (R^{i}\tilde{f}_{*}\overline{\mathbb{Q}}_{\ell})|\leq|\{\lambda\in\overline{\mathbb{A}}_{\mathbb{F}_{q}}^{1}:\tilde{f}\text
{ is not smooth in }\lambda\}|$.
Recall that $\tilde{X}_{\lambda}=V(G)\cap V(F-\lambda X_{0}^{d})$ (Lemma $\ref{lem : encaps}.(1)$). Since by hypothesis $V(G)\cap V(F-\lambda X_{0}^{d})\cap V(X_{0})$ is smooth, it is enough to bound the number of $\lambda$s such that the affine variety $V(g)\cap V(f-\lambda)$ is singular. The variety $V(g)\cap V(f-\lambda)$ is singular if the Jacobian matrix
\[
\begin{pmatrix}
f-\lambda  & g \\ \frac{\partial f}{\partial X_{1}} & \frac{\partial g}{\partial X_{1}} \\ \vdots & \vdots \\ \frac{\partial f}{\partial X_{n+1}} & \frac{\partial g}{\partial X_{n+1}}
\end{pmatrix}
\] 
has rank $\leq 1$ for some $\mathbf{v}:=(v_{1},...,v_{n+1})\in\mathbb{A}_{\mathbb{F}_{q}}^{n+1}$. Let us consider the polynomials $H_{\lambda,0}:=f-\lambda, H_{\lambda,n+1}:=g$, and for any $i=1,...,n$
\[
H_{\lambda,i}=\det
\begin{pmatrix}
\frac{\partial f}{\partial X_{i}} & \frac{\partial g}{\partial X_{i}} \\ \frac{\partial f}{\partial X_{i+1}} & \frac{\partial g}{\partial X_{i+1}}
\end{pmatrix}.
\]
If $V(g)\cap V(f-\lambda)$ is singular then there exists $\mathbf{v}\in\overline{\mathbb{A}}_{\mathbb{F}_{q}}^{n+1}$ such that $H_{\lambda,i}(\mathbf{v})=0$ for $i=0,...,n+1$. Notice that $H_{\lambda,i}$ are polynomial of degree at most $(e-1)(d-1)$ for $i=1,...,n$ and $\deg (H_{\lambda, 0})=\deg f$, $\deg (H_{\lambda, n+1})=\deg g$. Moreover, $H_{\lambda,i}$ does not depend on $\lambda$ for $i=1,...,n+1$. Recall that the resultant of $k+1$ polynomials in $k$ variables $f_{1},...,f_{k+1}$ of degree respectively $d_{1},...,d_{k+1}$ is an irreducible polynomial in the coefficients of $f_{1},...,f_{k+1}$ which vanishes if $f_{1},...,f_{k+1}$ have a common root. Using this we can conclude that if $V(g)\cap V(f-\lambda)$ is singular then $\text{Res}(H_{\lambda,0},...H_{\lambda,n+1})=0$. On the other hand, for any $i=1,...,n+1$ the coefficients of $H_{\lambda,i}$ are independent of $\lambda$ and the ones of $H_{\lambda,0}$ can be viewed as linear polynomials in $\lambda$, then $r(\lambda):=\text{Res}(H_{\lambda,0},...H_{\lambda,n+1})$ is a polynomial. Moreover $r(\lambda)$ is not the zero polynomial because $r(0)\neq 0$ by hypothesis ($V(G)\cap V(F)$ is smooth). Then $|\{\lambda\in\overline{\mathbb{A}}_{\mathbb{F}_{q}}^{1}:\tilde{f}\text{ is not smooth in }\lambda\}|\leq\deg (r(\lambda))\ll_{d,e,n} 1$, thanks to \cite[Chapter $13$, Proposition $1.1$]{GKZ08}.
\end{proof}
Let us now calculate $\rank (R^{i}\tilde{f}_{*}\overline{\mathbb{Q}}_{\ell})$.
\begin{lem}
For any $i\geq 0$ we have $\rank (R^{i}\tilde{f}_{*}\overline{\mathbb{Q}}_{\ell})\ll 1_{d,e,n}$.
\label{lem : rank}
\end{lem}
\begin{proof}
To prove the Lemma it is enough to calculate the dimension of the geometric fibers $(R^{i}\tilde{f}_{*}\overline{\mathbb{Q}}_{\ell})_{\lambda}$ when $\lambda$ is a lisse point of $R^{i}\tilde{f}_{*}\overline{\mathbb{Q}}_{\ell}$. We have already observed that $(R^{i}\tilde{f}_{*}\overline{\mathbb{Q}}_{\ell})_{\lambda }  \cong H_{c}^{i}(\tilde{X}_{\lambda},\overline{\mathbb{Q}}_{\ell})$. Thus, we need to compute $\dim (H_{c}^{i}(\tilde{X}_{\lambda},\overline{\mathbb{Q}}_{\ell}))$ for any $i$. We recall that if $\tilde{f}$ is smooth at $\lambda$, then $\tilde{X}_{\lambda}$ is a smooth variety of complete intersection of dimension $n-1$. Using the computations made in \cite[Theorem $8.1$]{Del74} we get
\begin{equation}
\dim (H_{c}^{i}(\tilde{X}_{\lambda},\overline{\mathbb{Q}}_{\ell}) =
\begin{cases}
 0 & \text{if $0\leq i \leq 2n-2$, $2\nmid i$ and $i\neq n-1$,} \\
 1 & \text{if $0\leq i \leq 2n-2$, $2 | i$ and $i\neq n-1$,} \\
 b_{n-1}(\tilde{X}_{\lambda}) -\frac{1+(-1)^{n-1}}{2} & \text{if $i=n-1$,}
\end{cases}
\label{eq : rank1}
\end{equation}
where $b_{n-1}(\tilde{X}_{\lambda})$ is the $(n-1)$-th Betti number which can be bounded in terms of $d,e,n$ only (\cite{Bom78}[Theorem $1$A]). Then $\rank (R^{i}\tilde{f}_{*}\overline{\mathbb{Q}}_{\ell}))\ll_{d,e,n}1$.
\end{proof}
\begin{lem}
For any $i\geq 0$ we have $\dim H_{c}^{0}(\overline{\mathbb{A}}_{\mathbb{F}_{q}}^{1},R^{i}\tilde{f}_{*}\overline{\mathbb{Q}}_{\ell})\ll_{d,e,n} 1$.
\label{lem : coh0}
\end{lem}
\begin{proof}
Using Remark $\ref{oss : zero}$, we have that $\dim H_{c}^{0}(\overline{\mathbb{A}}_{\mathbb{F}_{q}}^{1},R^{i}\tilde{f}_{*}\overline{\mathbb{Q}}_{\ell})\leq\sum_{\lambda\in \sing (R^{i}\tilde{f}_{*}\overline{\mathbb{Q}}_{\ell})}\dim ((R^{i}\tilde{f}_{*}\overline{\mathbb{Q}}_{\ell})_{\lambda})$. Then we need to bound $\dim ((R^{i}\tilde{f}_{*}\overline{\mathbb{Q}}_{\ell})_{\lambda})$ when $\lambda\in\sing (R^{i}\tilde{f}_{*}\overline{\mathbb{Q}}_{\ell})$. On the other hand, since $(R^{i}\tilde{f}_{*}\overline{\mathbb{Q}}_{\ell})_{\lambda }  \cong H_{c}^{i}(\tilde{X}_{\lambda},\overline{\mathbb{Q}}_{\ell})$, it is enough to show that that $\dim H_{c}^{i}(\tilde{X}_{\lambda},\overline{\mathbb{Q}}_{\ell})\ll_{d,e,n} 1$ in the case where $\tilde{X}_{\lambda}$ is a variety with at most isolated singularities (Proposition $\ref{prop : 1}.(ii)$). This is done, for example, in \cite[Appendix, Theorem $1$]{Kat91}. Hence
\[
\dim H_{c}^{0}(\overline{\mathbb{A}}_{\mathbb{F}_{q}}^{1},R^{i}\tilde{f}_{*}\overline{\mathbb{Q}}_{\ell})\leq |\sing ( R^{i}\tilde{f}_{*}\overline{\mathbb{Q}}_{\ell})|\cdot\max_{\lambda\in\sing(R^{i}\tilde{f}_{*}\overline{\mathbb{Q}}_{\ell}))}(\dim(R^{i}\tilde{f}_{*}\overline{\mathbb{Q}}_{\ell})_{\lambda })\ll_{d,e,n}1.
\]
\end{proof}
To conclude the proof of the Proposition, we need to bound the Swan conductors at singular points. To do this we first need the following
\begin{lem}
For any $i\geq 0$ we have that $\dim H_{c}^{1}(\overline{\mathbb{A}}_{\mathbb{F}_{q}}^{1},R^{i}\tilde{f}_{*}\overline{\mathbb{Q}}_{\ell})\ll_{d,e,n}1$.
\label{lem : coh1}
\end{lem}
\begin{proof}
The cohomology groups $H_{c}^{1}(\overline{\mathbb{A}}_{\mathbb{F}_{q}}^{1},R^{i}\tilde{f}_{*}\overline{\mathbb{Q}}_{\ell})$ are the starting objects for the Leray Spectral sequence arising from the map $\tilde{f}:\tilde{X}\longrightarrow\overline{\mathbb{A}}_{\mathbb{F}_{q}}^{1}$ and the $\ell$-adic sheaf $\overline{\mathbb{Q}}_{\ell}$ on $\overline{\mathbb{A}}_{\mathbb{F}_{q}}^{1}$, i.e. $H_{c}^{j}(\overline{\mathbb{A}}_{\mathbb{F}_{q}}^{1},R^{i}\tilde{f}_{*}\overline{\mathbb{Q}}_{\ell})=E_{2}^{j,i}\Rightarrow E^{i+j}=H_{c}^{i+j}(\tilde{X},\overline{\mathbb{Q}}_{\ell})$. On the other hand, $E_{\infty}^{1,i}=E_{2}^{1,i}= H_{c}^{1}(\overline{\mathbb{A}}_{\mathbb{F}_{q}}^{1},R^{i}\tilde{f}_{*}\overline{\mathbb{Q}}_{\ell})$ since $E_{r}^{1,i}= H_{c}^{1}(\overline{\mathbb{A}}_{\mathbb{F}_{q}}^{1},R^{i}\tilde{f}_{*}\overline{\mathbb{Q}}_{\ell})$ for any $r\geq 2$ because $H_{c}^{j}(\overline{\mathbb{A}}_{\mathbb{F}_{q}}^{1},R^{i}\tilde{f}_{*}\overline{\mathbb{Q}}_{\ell})=0$ if $j> 2$. Thus, one concludes that $\dim  H_{c}^{1}(\overline{\mathbb{A}}_{\mathbb{F}_{q}}^{1},R^{i}\tilde{f}_{*}\overline{\mathbb{Q}}_{\ell})\leq \dim H_{c}^{i+1}(\tilde{X},\overline{\mathbb{Q}}_{\ell})$. Hence, to conclude the proof it is enough to show that $\dim H_{c}^{k}(\tilde{X},\overline{\mathbb{Q}}_{\ell})\ll_{d,e,n}1$ for any $k\geq 0$. We start proving that $\dim ( H_{c}^{k}(\overline{V},\overline{\mathbb{Q}}_{\ell}))\ll_{d,e,n} 1$. Recall that $X=V\sqcup (X\cap L)$. Let us denote by $j:V\hookrightarrow X$ the open embedding of $V$ in $X$ and by $i: (X\cap L)\hookrightarrow X$ the closed embedding of $X\cap L$ in $X$. Then we have the exact sequence
\[
0\rightarrow j_{!}j^{*}\overline{\mathbb{Q}}_{\ell}\rightarrow \overline{\mathbb{Q}}_{\ell}\rightarrow i_{*}i^{*}\overline{\mathbb{Q}}_{\ell}\rightarrow 0,
\]
where $ \overline{\mathbb{Q}}_{\ell}$ denotes the trivial sheaf. This short exact sequence leads to the long exact sequence
\[
\cdots\rightarrow H_{c}^{i}(\overline{X},\overline{\mathbb{Q}}_{\ell})\rightarrow H_{c}^{i}(\overline{X\cap L},\overline{\mathbb{Q}}_{\ell})\rightarrow H_{c}^{i+1}(\overline{V},\overline{\mathbb{Q}}_{\ell})\rightarrow\cdots.
\] 
By hypothesis $X$ and $X\cap L$ are smooth complete intersection varieties, then $\dim (H_{c}^{i}(\overline{X},\overline{\mathbb{Q}}),\dim ( H_{c}^{i}(\overline{X\cap L},\overline{\mathbb{Q}}))\ll_{d,e,n}1$ for any $i$ thanks to the computation made in \cite[Theorem $8.1$]{Del74}. Hence, $\dim ( H_{c}^{i}(\overline{V},\overline{\mathbb{Q}}))\ll_{d,e,n}1$ using the exactness of the sequence above. Similarly, using the decomposition $\tilde {X}=V\sqcup \tilde{Z}$, we get the exact sequence
\[
\cdots\rightarrow H_{c}^{i}(\tilde{X},\overline{\mathbb{Q}}_{\ell})\rightarrow H_{c}^{i}(\overline{\tilde{Z}},\overline{\mathbb{Q}}_{\ell})\rightarrow H_{c}^{i+1}(\overline{V},\overline{\mathbb{Q}}_{\ell})\rightarrow\cdots.
\] 
Then the Lemma follows since $\dim ( H_{c}^{i}(\overline{V},\overline{\mathbb{Q}}))\ll_{d,e,n}1$, $\dim ( H_{c}^{i}(\overline{\tilde{Z}},\overline{\mathbb{Q}}))=\dim ( H_{c}^{i-2}(\overline{Z}),\overline{\mathbb{Q}})$ by K\"unneth formula (see \cite[VI.8]{Mil80}) and $\dim ( H_{c}^{i-2}(\overline{Z}),\overline{\mathbb{Q}})\ll_{d,e,n}1$.

\end{proof}
\begin{cor}
The Euler characteristic $\chi_{c}(\overline{\mathbb{A}}_{\mathbb{F}_{q}}^{1},R^{i}\tilde{f}_{*}\overline{\mathbb{Q}}_{\ell})\ll_{d,e,n} 1$.
\label{cor : euler}
\end{cor}
\begin{proof}
We have that
\[
\chi_{c}(\overline{\mathbb{A}}_{\mathbb{F}_{q}}^{1},R^{i}\tilde{f}_{*}\overline{\mathbb{Q}}_{\ell})=\dim (H_{c}^{0}(\overline{\mathbb{A}}_{\mathbb{F}_{q}}^{1},R^{i}\tilde{f}_{*}\overline{\mathbb{Q}}_{\ell}))-\dim (H_{c}^{1}(\overline{\mathbb{A}}_{\mathbb{F}_{q}}^{1},R^{i}\tilde{f}_{*}\overline{\mathbb{Q}}_{\ell}))+\dim (H_{c}^{2}(\overline{\mathbb{A}}_{\mathbb{F}_{q}}^{1},R^{i}\tilde{f}_{*}\overline{\mathbb{Q}}_{\ell})).
\]
Now $\dim (H_{c}^{0}(\overline{\mathbb{A}}_{\mathbb{F}_{q}}^{1},R^{i}\tilde{f}_{*}\overline{\mathbb{Q}}_{\ell})),\dim (H_{c}^{1}(\overline{\mathbb{A}}_{\mathbb{F}_{q}}^{1},R^{i}\tilde{f}_{*}\overline{\mathbb{Q}}_{\ell}))\ll_{d,e,n}1$, thanks to Lemma $\ref{lem : coh0}$ and Lemma $\ref{lem : coh1}$. Moreover, $\dim (H_{c}^{2}(\overline{\mathbb{A}}_{\mathbb{F}_{q}}^{1},R^{i}\tilde{f}_{*}\overline{\mathbb{Q}}_{\ell}))\leq \rank (R^{i}\tilde{f}_{*}\overline{\mathbb{Q}}_{\ell})\ll_{d,e,n}1$, thanks to Lemma $\ref{lem : rank}$.
\end{proof}
Finally, we can prove
\begin{cor}
For any $i\geq 0$ we have $\sum_{x}\Swan_{x}(j_{*}j^{*}R^{i}\tilde{f}_{*}\overline{\mathbb{Q}}_{\ell})\ll_{d,e,n}1$.
\label{cor : swan}
\end{cor}
\begin{proof}
Let $U\hookrightarrow\overline{\mathbb{A}}_{\mathbb{F}_{q}}^{1}$ be the largest open subset such that $R^{i}\tilde{f}_{*}\overline{\mathbb{Q}}_{\ell}$ is lisse, then one has
\[
\chi_{c}(\overline{\mathbb{A}}_{\mathbb{F}_{q}}^{1},R^{i}\tilde{f}_{*}\overline{\mathbb{Q}}_{\ell})=\chi_{c}(\overline{U},R^{i}\tilde{f}_{*}\overline{\mathbb{Q}}_{\ell})+\sum_{s\in\sing (R^{i}\tilde{f}_{*}\overline{\mathbb{Q}}_{\ell})}\dim ((R^{i}\tilde{f}_{*}\overline{\mathbb{Q}}_{\ell})_{s}).
\]
Thus, $\chi_{c}(\overline{U},R^{i}\tilde{f}_{*}\overline{\mathbb{Q}}_{\ell})\ll_{d,e,n} 1$, thanks to Lemma $\ref{lem : coh0}$ and Corollary $\ref{cor : euler}$. On the other hand, using the Grothendieck–Ogg–Shafarevich Formula we get
\[
\chi_{c}(\overline{U},R^{i}\tilde{f}_{*}\overline{\mathbb{Q}}_{\ell})= (2-|\sing (R^{i}\tilde{f}_{*}\overline{\mathbb{Q}}_{\ell})|)\cdot\rank (R^{i}\tilde{f}_{*}\overline{\mathbb{Q}}_{\ell})-\sum_{x}\Swan_{x}(j_{*}j^{*}R^{i}\tilde{f}_{*}\overline{\mathbb{Q}}_{\ell}),
\]
then the Corollary follows from Lemma $\ref{lem : sing}$, Lemma $\ref{lem : rank}$.
\end{proof}
Using Lemma $\ref{lem : sing}$, Lemma $\ref{lem : rank}$, Lemma $\ref{lem : coh0}$ and Corollary $\ref{cor : swan}$ we can bound any quantity appearing in $c(R^{i}\tilde{f}_{*}\overline{\mathbb{Q}}_{\ell})$ in terms of $d,e$ and $n$ only and this conclude the proof of the proposition.

\end{proof}

\subsubsection*{Frobenius action}

\begin{lem}
For any even $i\geq n+1$ one has that:
\[
t_{R^{i}\tilde{f}_{*}\overline{\mathbb{Q}}_{\ell}}(\lambda)=q^{\frac{i}{2}},
\]
for any $\lambda\in\overline{\mathbb{A}}_{\mathbb{F}_{q}}^{1}(\mathbb{F}_{q})$. Moreover
\[
t_{j_{*}j^{*} R^{n}\tilde{f}_{*}\overline{\mathbb{Q}}_{\ell}}(\lambda)=\frac{q^{\frac{n}{2}}(1+(-1)^{n})}{2},
\]
for any $\lambda\in\overline{\mathbb{A}}_{\mathbb{F}_{q}}^{1}(\mathbb{F}_{q})$.
\end{lem}
\begin{proof}
By Corollary $\ref{cor : geomconst}$, we know that for any $i\geq n+1$ even the sheaf $R^{i}\tilde{f}_{*}\overline{\mathbb{Q}}_{\ell}$ is geometrically irreducible and geometrically constant, i.e. $R^{i}\tilde{f}_{*}\overline{\mathbb{Q}}_{\ell}\cong\chi_{i}\otimes\overline{\mathbb{Q}}_{\ell}$ where $\chi_{i}:\Gal\big(\overline{\mathbb{F}}_{q}/\mathbb{F}_{q}\big)\longrightarrow\overline{\mathbb{Q}}_{\ell}^{\times}$ is a character. Hence, $t_{R^{i}\tilde{f}_{*}\overline{\mathbb{Q}}_{\ell}}(\lambda)=\alpha_{i}$, where $\alpha_{i}=\chi_{i} (\Fr^{\geom})$ and $\alpha_{i}$ is a $q$-Weil number of weight $\leq i$ (\cite[Theorem $1$]{Del80}). Applying the Grothendieck-Lefschetz trace formula (\cite[Expose VI]{SGA4.5}) one has that
\[
\alpha_{i} q=\sum_{x\in\mathbb{F}_{q}}t_{\chi_{i}\otimes\overline{\mathbb{Q}}_{\ell}}(x)=\Tr(\Fr|H_{c}^{2}(\overline{\mathbb{A}}_{\mathbb{F}_{q}}^{1},\chi_{i}\otimes\overline{\mathbb{Q}}_{\ell})).
\]
For the sheaf $j_{*}j^{*} R^{n}\tilde{f}_{*}\overline{\mathbb{Q}}_{\ell}$ we have to distinguish two cases:
\begin{itemize}
\item[$(i)$] $n$ odd. In this case one has $\rank (R^{n}\tilde{f}_{*}\overline{\mathbb{Q}}_{\ell})=0$, then $\rank (j_{*}j^{*} R^{n}\tilde{f}_{*}\overline{\mathbb{Q}}_{\ell})=0$ and this implies $t_{j_{*}j^{*} R^{n}\tilde{f}_{*}\overline{\mathbb{Q}}_{\ell}}(\lambda)=0$ for any $\lambda\in\overline{\mathbb{A}}_{\mathbb{F}_{q}}^{1}$.
\item[$(ii)$] $n$ even. By Proposition $\ref{prop : 1}$, the sheaf $j_{*}j^{*} R^{n}\tilde{f}_{*}\overline{\mathbb{Q}}_{\ell}$ is geometrically constant of rank $1$. Then $j_{*}j^{*}R^{n}\tilde{f}_{*}\overline{\mathbb{Q}}_{\ell}\cong\chi_{n}\otimes\overline{\mathbb{Q}}_{\ell}$ for some character $\chi_{n}:\Gal\big(\overline{\mathbb{F}}_{q}/\mathbb{F}_{q}\big)\longrightarrow\overline{\mathbb{Q}}_{\ell}^{\times}$. Hence, $
t_{j_{*}j^{*}R^{n}\tilde{f}_{*}\overline{\mathbb{Q}}_{\ell}}(\lambda)=\alpha_{n}$, with $\alpha_{n}=\chi_{n}(\Fr^{\geom})$, and $\alpha_{n}$ is a $q$-Weil number of weight $\leq n$ (\cite[Theorem $1$]{Del80}).
\end{itemize}
We can state both cases by writing
\[
t_{j_{*}j^{*}R^{n}\tilde{f}_{*}\overline{\mathbb{Q}}_{\ell}}(\lambda)=\frac{\alpha_{n}(1+(-1)^{n})}{2}.
\]
By the Grothendieck-Lefschetz trace formula we get
\[
\frac{q\alpha_{n}(1+(-1)^{n})}{2}=\sum_{x\in\mathbb{F}_{q}}t_{\chi_{n}\otimes\overline{\mathbb{Q}}_{\ell}}(x)=\Tr(\Fr|H_{c}^{2}(\overline{\mathbb{A}}_{\mathbb{F}_{q}}^{1},\chi_{n}\otimes\overline{\mathbb{Q}}_{\ell})).
\]
This shows that in order to compute $\alpha_{i}$ it is enough compute $\Tr(\Fr|H_{c}^{2}(\overline{\mathbb{A}}_{\mathbb{F}_{q}}^{1},R^{i}\tilde{f}_{*}\overline{\mathbb{Q}}_{\ell})$ and $\Tr(\Fr|H_{c}^{2}(\overline{\mathbb{A}}_{\mathbb{F}_{q}}^{1},j_{*}j^{*} R^{n}\tilde{f}_{*}\overline{\mathbb{Q}}_{\ell}))$. Using \cite[Lemma $11$]{Kat02a}, we can write
\[
|\tilde{X}(\mathbb{F}_{q})|=\sum_{i}(-1)^{i}\sum_{j=0}^{2}(-1)^{j}\Tr(\Fr|H_{c}^{j}(\overline{\mathbb{A}}^{1}_{\mathbb{F}_{q}},R^{i}\tilde{f}_{*}\overline{\mathbb{Q}}_{\ell})).
\]
Then using the fact that for $i\geq n+1$ the sheaves $R^{i}\tilde{f}_{*}\overline{\mathbb{Q}}_{\ell}$ are geometrically constant, that  $R^{i}\tilde{f}_{*}\overline{\mathbb{Q}}_{\ell}=0$ for $i$ odd, and applying Lemma $\ref{lem : 4}$ one gets
\[
\begin{split}
|\tilde{X}(\mathbb{F}_{q})| &=\sum_{i\text{ even }, i\geq n+1}\Tr(\Fr|H_{c}^{2}(\overline{\mathbb{A}}^{1}_{\mathbb{F}_{q}},R^{i}\tilde{f}_{*}\overline{\mathbb{Q}}_{\ell}))\\&+(-1)^{n}(\Tr(\Fr|H_{c}^{0}(\overline{\mathbb{A}}^{1}_{\mathbb{F}_{q}},\mathcal{P}))+\Tr(\Fr|H_{c}^{2}(\overline{\mathbb{A}}^{1}_{\mathbb{F}_{q}},j_{*}j^{*} R^{n}\tilde{f}_{*}\overline{\mathbb{Q}}_{\ell}))\\& +\sum_{i< n}(-1)^{j}\sum_{j=0}^{2}(-1)^{j}\Tr(\Fr|H_{c}^{j}(\overline{\mathbb{A}}^{1}_{\mathbb{F}_{q}},R^{i}\tilde{f}_{*}\overline{\mathbb{Q}}_{\ell})).
\end{split}
\]
Applying Lemma $\ref{lem : 2}$ and \cite[Theorem $1$, Theorem $2$]{Del80} one gets
\begin{equation}
|\tilde{X}(\mathbb{F}_{q})|=\frac{q\alpha_{n}(1+(-1)^{n})}{2}+\sum_{i= n+1,i\text{ even }}^{2n-2}q\alpha_{i}+ O_{d,n}(q^{\frac{n+1}{2}}),
\label{eq : cost1}
\end{equation}
where for any $i\geq n$ in the above sum, $|q\alpha_{i}|= q^{\frac{i}{2}+1}$. On the other hand, we can compute $|\tilde{X}(\mathbb{F}_{q})|$ using the decomposition $\tilde{X}=V\sqcup \tilde{Z}$. Indeed, we have that $|V(\mathbb{F}_{q})|=q^{n}+O_{d,n}(q^{\frac{n}{2}})$ thanks to Proposition $\ref{prop : coundel}$ and that
\[
|\tilde{Z}(\mathbb{F}_{q})|=q|Z(\mathbb{F}_{q})|=\frac{q^{\frac{n}{2}+1}(1+(-1)^{n})}{2}+ \sum_{b=n+1, b \text{ even }}^{b=2n-2}q^{\frac{b}{2}}+O_{d,n}(q^{\frac{n+1}{2}}).
\]
thanks to \cite{Del74}[Theorem $8.1$]. Thus, we get
\begin{equation}
|\tilde{X}(\mathbb{F}_{q})|=\frac{q^{\frac{n}{2}+1}(1+(-1)^{n})}{2}+\sum_{b=n+3, b \text{ even }}^{b=2n}q^{\frac{b}{2}}+O_{d,n}(q^{\frac{n+1}{2}}).
\label{eq : cost2}
\end{equation}
Comparing the right hand side of ($\ref{eq : cost1}$) with the one of ($\ref{eq : cost2}$) (replacing $\mathbb{F}_{q}$ by a suitable extension $\mathbb{F}_{q^{\nu}}$ if necessary) we obtain that $q\alpha_{i}= q^{\frac{i}{2}+1}$. Hence, $\alpha_{i}=q^{\frac{i}{2}}$ for any $i\geq n$ as we want.
\end{proof}
\begin{cor}
For any $i\geq n+1$, one has that
\[
\sum_{k=0}^{2}(-1)^{k}\Tr (\Fr |H_{c}^{k}(\overline{\mathbb{A}}^{1}_{\mathbb{F}_{q}},\mathcal{F}\otimes R^{i}\tilde{f}_{*}\overline{\mathbb{Q}}_{\ell}))=
\begin{cases}
q^{\frac{i}{2}}\sum_{x\in\mathbb{F}_{q}}t_{\mathcal{F}}(x),    &\text{if $i$ is even}\\
0 &\text{if $i$ odd}.
\end{cases}
\]
Moreover
\[
\sum_{j=0}^{2}(-1)^{k}\Tr (\Fr |H_{c}^{j}(\overline{\mathbb{A}}^{1}_{\mathbb{F}_{q}},\mathcal{F}\otimes R^{n}\tilde{f}_{*}\overline{\mathbb{Q}}_{\ell}))=\frac{q^{\frac{n}{2}}(1+(-1)^{n})}{2}\sum_{x\in\mathbb{F}_{q}}t_{\mathcal{F}}(x)+O_{e,d,n,c(\mathcal{F})}(q^{\frac{n}{2}}).
\]
\label{cor : frobac}
\end{cor}
\begin{proof}
Applying the Lefschetz trace formula one gets
\[
\sum_{k=0}^{2}(-1)^{k}\Tr (\Fr |H_{c}^{k}(\overline{\mathbb{A}}^{1}_{\mathbb{F}_{q}},\mathcal{F}\otimes R^{i}\tilde{f}_{*}\overline{\mathbb{Q}}_{\ell}))=\sum_{x\in \mathbb{F}_{q}}t_{\mathcal{F}\otimes R^{i}\tilde{f}_{*}\overline{\mathbb{Q}}_{\ell}}(x).
\]
For $i$ odd there is nothing to prove since $H_{c}^{k}(\overline{\mathbb{A}}^{1}_{\mathbb{F}_{q}},\mathcal{F}\otimes R^{i}\tilde{f}_{*}\overline{\mathbb{Q}}_{\ell})=0$ for $k=0,1,2$. Otherwise $\mathcal{F}\otimes R^{i}\tilde{f}_{*}\overline{\mathbb{Q}}_{\ell}\cong_{\geom}\mathcal{F}$ and this implies that $t_{\mathcal{F}\otimes R^{i}\tilde{f}_{*}\overline{\mathbb{Q}}_{\ell}}(x)=\alpha_{i}t_{\mathcal{F}}(x)$ for any $x\in\mathbb{F}_{q}$. On the other hand, we have shown in the previous Lemma that $\alpha_{i}=q^{\frac{i}{2}}$ and the result follows. For the second part of the argument one starts writing
\[
\sum_{x\in \mathbb{F}_{q}}t_{\mathcal{F}\otimes R^{n}\tilde{f}_{*}\overline{\mathbb{Q}}_{\ell}}(x)=\sum_{x\in \mathbb{F}_{q}\setminus\sing (R^{n}\tilde{f}_{*}\overline{\mathbb{Q}}_{\ell})(\mathbb{F}_{q})}t_{\mathcal{F}\otimes R^{i}\tilde{f}_{*}\overline{\mathbb{Q}}_{\ell}}(x)+\sum_{x\in\sing (R^{n}\tilde{f}_{*}\overline{\mathbb{Q}}_{\ell})(\mathbb{F}_{q})}t_{\mathcal{F}\otimes R^{i}\tilde{f}_{*}\overline{\mathbb{Q}}_{\ell}}(x).
\]
On the other hand, thanks to Lemma $\ref{lem : sing}$ together with \cite{Del80}[Theorem $1$] one has that the second second sum in the equation above is $\ll_{d,e,n}q^{\frac{n}{2}}$ . To conclude the proof it is enough to observe that $t_{\mathcal{F}\otimes R^{n}\tilde{f}_{*}\overline{\mathbb{Q}}_{\ell}}=t_{\mathcal{F}\otimes j_{*}j^{*}R^{n}\tilde{f}_{*}\overline{\mathbb{Q}}_{\ell}}$ on $\mathbb{F}_{q}\setminus\sing (R^{n}\tilde{f}_{*}\overline{\mathbb{Q}}_{\ell}(\mathbb{F}_{q}))$ and apply the same argument as above to $t_{\mathcal{F}\otimes j_{*}j^{*}R^{n}\tilde{f}_{*}\overline{\mathbb{Q}}_{\ell}}$.
\end{proof}
\begin{lem}
One has
\[
\sum_{i=0}^{2(n-1)}(-1)^{i}\Tr (\Fr|H_{c}^{i}(\overline{\tilde{Z}}, \tilde{f}^{*}\mathcal{F}_{|Z}))=\Big(\sum_{x\in\mathbb{F}_{q}}t_{\mathcal{F}}(x)\Big)\Big(\sum_{i=n, i\text{ even}}^{i=2n-4}q^{\frac{i}{2}}\Big)+O_{d,e,n}(q^{\frac{n}{2}}).
\]
\label{lem : 5}
\end{lem}
\begin{proof}
The action of the Frobenius on the cohomology groups $H_{c}^{i}(\overline{\tilde{Z}}, \tilde{f}^{*}\mathcal{F}_{|Z})$ can be calculated by observing that $\tilde{Z}=Z\times\overline{\mathbb{A}}_{\mathbb{F}_{q}}^{1}$. Indeed, using the K\"unneth formula (see \cite[VI.8]{Mil80}) one gets
\[
H_{c}^{i}(\overline{\tilde{Z}}, \tilde{f}^{*}\mathcal{F}_{|Z})=\bigoplus_{b=0}^{1}H_{c}^{i-b}(\overline{Z}, \overline{\mathbb{Q}}_{\ell})\otimes H_{c}^{b}(\overline{\mathbb{A}}^{1}_{\mathbb{F}_{q}},\mathcal{F}).
\]
Combining this with the functoriality of the Frobenius one obtains
\[
\Tr(\Fr|H_{c}^{i}(\overline{Z}, \tilde{f}^{*}\mathcal{F}_{|Z}))=\sum_{b=0}^{1}\Tr(\Fr|H_{c}^{i-b}(\overline{Z}, \overline{\mathbb{Q}}_{\ell}))\Tr(\Fr|H_{c}^{b}(\overline{\mathbb{A}}^{1}_{\mathbb{F}_{q}},\mathcal{F})).
\]
Thus,
\[
\begin{split}
\sum_{i=0}^{2(n-1)}(-1)^{i}\Tr (\Fr|H_{c}^{i}(\overline{\tilde{Z}}, \tilde{f}^{*}\mathcal{F}_{|Z}))&=\sum_{i=0}^{2(n-1)}(-1)^{i}\sum_{b=0}^{1}\Tr(\Fr|H_{c}^{i-b}(\overline{Z}, \overline{\mathbb{Q}}_{\ell}))\Tr(\Fr|H_{c}^{b}(\overline{\mathbb{A}}^{1}_{\mathbb{F}_{q}},\mathcal{F}))\\&=\sum_{i=0}^{2(n-1)}(-1)^{i}\Tr(\Fr|H_{c}^{i}(\overline{Z}, \overline{\mathbb{Q}}_{\ell}))\sum_{b=0}^{1}(-1)^{b}\Tr(\Fr|H_{c}^{b}(\overline{\mathbb{A}}^{1}_{\mathbb{F}_{q}},\mathcal{F})).
\end{split}
\]
On the other hand $\sum_{x\in\mathbb{F}_{q}}t_{\mathcal{F}}(x)=\sum_{b=0}^{1}(-1)^{b}\Tr(\Fr|H_{c}^{b}(\overline{\mathbb{A}}^{1}_{\mathbb{F}_{q}},\mathcal{F}))$ (we are assuming that $H_{c}^{2}(\overline{\mathbb{A}}^{1}_{\mathbb{F}_{q}},\mathcal{F})=0$), and
\[
\Tr(\Fr|H_{c}^{i}(\overline{Z}, \overline{\mathbb{Q}}_{\ell})) =
\begin{cases}
 0 & \text{if either $0\leq i \leq 2n-4$, $2\nmid i$ and $i\neq n-1$, or $i>2n-4$} \\
 q^{i} & \text{if $0\leq i \leq 2n-4$, $2 | i$ and $i\neq n-1$,}
\end{cases}
\]
thanks to \cite{Del74}[Theorem $8.1$]. Moreover $|\Tr(\Fr|H_{c}^{n-2}(\overline{Z}, \overline{\mathbb{Q}}_{\ell}))|\ll_{n,e,d} q^{\frac{n-2}{2}}$ again thanks to \cite{Del74}[Theorem $8.1$].
\end{proof}
\subsubsection*{End of the proof of Theorem $\ref{thm : 1}$}
We are finally ready to prove Theorem $\ref{thm : 1}$. First of all, observe that
\[
\sum_{x\in\tilde{X}(\mathbb{F}_{q})}t_{\tilde{f}^{*}\mathcal{F}}(x)=\sum_{x\in V(\mathbb{F}_{q})}t_{f^{*}\mathcal{F}}(x)+\sum_{x\in Z(\mathbb{F}_{q})}t_{\tilde{f}^{*}\mathcal{F}_{|Z}}(x).
\]
Arguing as \cite[Lemma $11$]{Kat99}, one has
\[
\begin{split}
\sum_{x\in\tilde{X}(\mathbb{F}_{q})}t_{\tilde{f}^{*}\mathcal{F}}(x)&=\sum_{i,j}(-1)^{i+j}\Tr(\Fr |H_{c}^{j}(\overline{\mathbb{A}}^{1}_{\mathbb{F}_{q}},\mathcal{F}\otimes R^{i}\tilde{f}_{*}\overline{\mathbb{Q}}_{\ell}))\\&=\sum_{i}(-1)^{i}\sum_{j=0}^{2}(-1)^{j}\Tr(\Fr |H_{c}^{j}(\overline{\mathbb{A}}^{1}_{\mathbb{F}_{q}},\mathcal{F}\otimes R^{i}\tilde{f}_{*}\overline{\mathbb{Q}}_{\ell})).
\end{split}
\]
If $i> n$ and $i$ is even, then we apply Corollary $\ref{cor : frobac}$ getting
\[
\sum_{j=0}^{2}(-1)^{j}\Tr(\Fr |H_{c}^{j}(\overline{\mathbb{A}}^{1}_{\mathbb{F}_{q}},\mathcal{F}\otimes R^{i}\tilde{f}_{*}\overline{\mathbb{Q}}_{\ell}))=q^{\frac{i}{2}}\sum_{x\in\mathbb{F}_{q}}t_{\mathcal{F}}(x).
\]
Similarly, for $i=n$ one gets
\[
\sum_{j=0}^{2}(-1)^{j}\Tr(\Fr |H_{c}^{j}(\overline{\mathbb{A}}^{1}_{\mathbb{F}_{q}},\mathcal{F}\otimes R^{n}\tilde{f}_{*}\overline{\mathbb{Q}}_{\ell}))=\frac{q^{\frac{i}{2}}(1+(-1)^{n})}{2}\sum_{x\in\mathbb{F}_{q}}t_{\mathcal{F}}(x)+O_{d,e,c(\mathcal{F})}(q^{\frac{n}{2}}).
\]
Otherwise, the Riemann Hypothesis over finite fields implies that
\[
\sum_{j=0}^{2}(-1)^{j}\Tr(\Fr |H_{c}^{j}(\overline{\mathbb{A}}^{1}_{\mathbb{F}_{q}},\mathcal{F}\otimes R^{i}\tilde{f}_{*}\overline{\mathbb{Q}}_{\ell}))=\sum_{x\in\mathbb{F}_{q}}t_{\mathcal{F}\otimes R^{i}\tilde{f}_{*}\overline{\mathbb{Q}}_{\ell}}(x)\leq c(\mathcal{F}\otimes R^{i}\tilde{f}_{*}\overline{\mathbb{Q}}_{\ell}))q^{\frac{i+1}{2}}\ll_{d,e,n,c(\mathcal{F})}q^{\frac{i+1}{2}},
\]
where in the last step we used Lemma $\ref{lem : cohom2}$ and Lemma $\ref{lem : 2}$. Thus
\begin{equation}
\sum_{x\in\tilde{X}(\mathbb{F}_{q})}t_{\tilde{f}^{*}\mathcal{F}}(x)=\Big(\sum_{x\in\mathbb{F}_{q}}t_{\mathcal{F}}(x)\Big)\Big(\sum_{i=n, i\text{ even}}^{i=2n-2}q^{\frac{i}{2}}\Big)+O_{d,n,c(\mathcal{F})}(q^{\frac{n}{2}}).
\label{eq : trace1}
\end{equation}
On the other hand, combining the Lefschetz trace formula for $\tilde{f}^{*}\mathcal{F}_{|{Z}}$ on $Z$ with Lemma $\ref{lem : 5}$ we get
\begin{equation}
\sum_{x\in Z(\mathbb{F}_{q})}t_{\tilde{f}^{*}\mathcal{F}_{|Z}}(x)=\Big(\sum_{x\in\mathbb{F}_{q}}t_{\mathcal{F}}(x)\Big)\Big(\sum_{i=n, i\text{ even}}^{i=2n-4}q^{\frac{i}{2}}\Big)+O_{d,n,c(\mathcal{F})}(q^{\frac{n}{2}}),
\label{eq : trace2}
\end{equation}
subtracting ($\ref{eq : trace2}$) to $(\ref{eq : trace1})$ we get:
\[
\sum_{x\in V(\mathbb{F}_{q})}t_{f^{*}\mathcal{F}}(x)=q^{n-1}\sum_{x\in\mathbb{F}_{q}}t_{\mathcal{F}}(x)+O_{d,n,c(\mathcal{F})}(q^{\frac{n}{2}}),
\]
as we want. 

\subsection*{Proof of Theorem $\ref{cor : mix}$}
\label{subsec : mix}
We start writing
\begin{equation}
\sum_{\mathbf{x}\in\mathbb{F}_{q}^{n}}t_{\mathcal{F}}(F(\mathbf{x}))\psi(G(\mathbf{x}))=\sum_{(a,b)\in\mathbb{F}_{q}^{2}}t_{\mathcal{F}}(a)\psi(b)N(a,b,F,G),
\label{eq : starting}
\end{equation}
where for any $a,b\in\mathbb{F}_{q}$, $N(a,b,F,G):=|\{\mathbf{x}\in\mathbb{F}_{q}^{n}:F(\mathbf{x})=a\text{ and }G(\mathbf{x})=b\}|.$
At this point it is useful to
\begin{itemize}
\item[$(i)$] recall that for any $a,b\in\mathbb{F}_{q}$ one has that $N(a,b,F,G)=q^{n-2}+O_{d,e,n}(q^{\frac{n-2}{2}})$ if $V(F-a)\cap V(G-b)$ (Proposition $\ref{prop : coundel}$) is a smooth variety, and $N(a,b,F,G)=q^{n-2}+O_{d,e,n}(q^{\frac{n-1}{2}})$ if $V(F-a)\cap V(G-b)$ is singular (the proof of this is similar to the one of Proposition $\ref{prop : coundel}$),
\item[$(ii)$] observe that if $a\neq 0$
\begin{equation}
\sum_{b\in\mathbb{F}_{q}}N(a,b,F,G)=N(a,F)=q^{n-1}+O_{d,n}(q^{\frac{n-1}{2}})
\label{eq : sumdel}
\end{equation}
where $N(a,F):=|\{\mathbf{x}\in\mathbb{F}_{q}^{n}:F(\mathbf{x})=a\}|$.
\end{itemize}
Since $F,G$ are homogeneous, for $a,b\in\mathbb{F}_{q}$ and $\eta\in\mathbb{F}_{b}^{\times}$ we have $N(a,b,F,G)=N(\eta^{d}a,\eta^{e} b,F,G)$: this can be proven by using the transformation $(x_{i})\mapsto (\eta x_{i})$. On the other hand, the morphism
\[
\begin{matrix}
\varphi:\text{ } \mathbb{F}_{q}^{\times}/\mathbb{F}_{q}^{\times d}\times\mathbb{F}_{q}^{\times}\times\mathbb{F}_{q}^{\times} & \longrightarrow & \mathbb{F}_{q}^{\times}\times\mathbb{F}_{q}^{\times}\\(\alpha,\eta, b) & \longmapsto & (\eta^{d}\alpha,\eta^{e}b)
\end{matrix}
\]
is a surjection onto $\mathbb{F}_{q}^{\times}\times\mathbb{F}_{q}^{\times}$ with $|\text{ker}(\varphi)|=(d,q-1)$. Thus we can rewrite $(\ref{eq : starting})$ as
\begin{equation}
\begin{split}
\sum_{\mathbf{x}\in\mathbb{F}_{q}^{n}}t_{\mathcal{F}}(F(\mathbf{x}))\psi(G(\mathbf{x}))&=\sum_{(a,b)\in\mathbb{F}_{q}^{2}}t_{\mathcal{F}}(a)\psi(b)N(a,b,F,G)\\&=\sum_{a\in\mathbb{F}_{q}}t_{\mathcal{F}}(a)N(a,0,F,G)+\sum_{b\in\mathbb{F}_{q}}t_{\mathcal{F}}(0)\psi (b)N(0,b,F,G)\\&+\frac{1}{(d,q-1)}\sum_{\alpha\in\mathbb{F}_{q}^{\times}/\mathbb{F}_{q}^{\times d}}\sum_{b\in\mathbb{F}_{q}^{\times}}N(\alpha,b,F,G)\sum_{\eta\in\mathbb{F}_{q}^{\times}}t_{\mathcal{F}} (\alpha\eta^{d})\psi(b\eta^{e})\\& -t_{\mathcal{F}}(0)N(0,0,F,G).
\end{split}
\label{eq : id}
\end{equation}
Now we remove the condition $\eta\in\mathbb{F}_{q}^{\times}$ in the last sum. To do so, we observe that 
\[
\begin{split}
\frac{1}{(d,q-1)}\sum_{\alpha\in\mathbb{F}_{q}^{\times}/\mathbb{F}_{q}^{\times d}}\sum_{b\in\mathbb{F}_{q}^{\times}}N(\alpha,b,F,G)t_{\mathcal{F}}(0)&=\frac{1}{(d,q-1)}\sum_{\alpha\in\mathbb{F}_{q}^{\times}/\mathbb{F}_{q}^{\times d}}\sum_{b\in\mathbb{F}_{q}}N(\alpha,b,F,G)t_{\mathcal{F}}(0)\\&-\frac{1}{(d,q-1)}\sum_{\alpha\in\mathbb{F}_{q}^{\times}/\mathbb{F}_{q}^{\times d}}N(\alpha,0,F,G)t_{\mathcal{F}}(0)\\&=\frac{1}{(d,q-1)}\sum_{\alpha\in\mathbb{F}_{q}^{\times}/\mathbb{F}_{q}^{\times d}}N(\alpha,F)t_{\mathcal{F}}(0)\\&-\frac{1}{(d,q-1)}\sum_{\alpha\in\mathbb{F}_{q}^{\times}/\mathbb{F}_{q}^{\times d}}N(\alpha,0,F,G)t_{\mathcal{F}}(0).
\end{split}
\]
where in the last step we used $(\ref{eq : sumdel})$. On the other hand we have
\[
N(\alpha, F)=q^{n-1}+O_{d,e,n}(q^{\frac{n-1}{2}}),\qquad N(\alpha,0,F, G)=q^{n-2}+O_{d,e,n}(q^{\frac{n}{2}-1})
\] 
for any $\alpha\in\mathbb{F}_{q}^{\times}/\mathbb{F}_{q}^{\times d}$. Thus
\[
\frac{1}{(d,q-1)}\sum_{\alpha\in\mathbb{F}_{q}^{\times}/\mathbb{F}_{q}^{\times d}}\sum_{b\in\mathbb{F}_{q}^{\times}}N(\alpha,b,F,G)t_{\mathcal{F}}(0)=q^{n-2}(q-1)t_{\mathcal{F}}(0)+O_{d,e,n,c(\mathcal{F})}(q^{\frac{n}{2}}).
\]
Moreover $N(0,0,F,G)=q^{n-2}+O_{d,e,n}(q^{\frac{n-1}{2}})$ because by hypothesis the affine variety $\{\mathbf{x}\in\mathbb{F}_{q}^{n}:F(\mathbf{x})=0\text{ and }G(\mathbf{x})=0\}$ is singular only at the origin. So we can rewrite $(\ref{eq : id})$ as
\[
\begin{split}
\sum_{\mathbf{x}\in\mathbb{F}_{q}^{n}}t_{\mathcal{F}}(F(\mathbf{x}))\psi(G(\mathbf{x}))&=\sum_{(a,b)\in\mathbb{F}_{q}^{2}}t_{\mathcal{F}}(a)\psi(b)N(a,b,F,G)\\&=\sum_{a\in\mathbb{F}_{q}}t_{\mathcal{F}}(a)N(a,0,F,G)+\sum_{b\in\mathbb{F}_{q}}t_{\mathcal{F}}(0)\psi (b)N(0,b,F,G)\\&+\frac{1}{(d,q-1)}\sum_{\alpha\in\mathbb{F}_{q}^{\times}/\mathbb{F}_{q}^{\times d}}\sum_{b\in\mathbb{F}_{q}^{\times}}N(\alpha,b,F,G)\sum_{\eta\in\mathbb{F}_{q}}t_{\mathcal{F}} (\alpha\eta^{d})\psi(b\eta^{e})\\& +E(q).
\end{split}
\]
where $E(q)=-q^{n-1}t_{\mathcal{F}}(0)+O_{d,e,n,c(\mathcal{F})}(q^{\frac{n}{2}})$. Let us discuss first 
\[
M:=\frac{1}{(d,q-1)}\sum_{\alpha\in\mathbb{F}_{q}^{\times}/\mathbb{F}_{q}^{\times d}}\sum_{b\in\mathbb{F}_{q}^{\times}}N(\alpha,b,F,G)\sum_{\eta\in\mathbb{F}_{q}}t_{\mathcal{F}} (\alpha\eta^{d})\psi(b\eta^{e})
\]
To simplify the notation we will denote $\mathcal{G}_{\alpha,e}:=[\times(-1)]^{*}T_{e}([\times\alpha]^{*}[x\mapsto x^{d}]^{*}\mathcal{F})$ (see Definition $\ref{defn : power}$). Observe that
\[
t_{\mathcal{G}_{\alpha,e}}(b)=\frac{1}{\sqrt{q}}\sum_{\eta\in\mathbb{F}_{q}}t_{\mathcal{F}} (\alpha\eta^{d})\psi(b\eta^{e})
\]
for any $b\in\mathbb{F}_{q}$. Then $M$ become 
\begin{equation}
\begin{split}
M&=\frac{\sqrt{q}}{(d,q-1)}\sum_{\alpha\in\mathbb{F}_{q}^{\times}/\mathbb{F}_{q}^{\times d}}\sum_{b\in\mathbb{F}_{q}^{\times}}N(\alpha,b,F,G)t_{\mathcal{G}_{e,\alpha}}(b)\\&=\frac{\sqrt{q}}{(d,q-1)}\sum_{\alpha\in\mathbb{F}_{q}^{\times}/\mathbb{F}_{q}^{\times d}}\sum_{\mathbf{x} : \substack{F(\mathbf{x})=\alpha\\G(\mathbf{x})\neq 0}}t_{\mathcal{G}_{\alpha,e}}(G(\mathbf{x}))\\&=\frac{\sqrt{q}}{(d,q-1)}\sum_{\alpha\in\mathbb{F}_{q}^{\times}/\mathbb{F}_{q}^{\times d}}\sum_{\mathbf{x} : F(\mathbf{x})=\alpha}t_{\mathcal{G}_{\alpha,e}}(G(\mathbf{x}))\\&-\frac{\sqrt{q}}{(d,q-1)}\sum_{\alpha\in\mathbb{F}_{q}^{\times}/\mathbb{F}_{q}^{\times d}}N(\alpha,0,F,G)t_{\mathcal{G}_{\alpha,e}}(0).
\end{split}
\end{equation}
On the other hand one has that $\mathcal{G}_{\alpha,e}$ is irreducible and not trivial since $\mathcal{G}_{1,e}=[\times(-1)]^{*}T_{e}([x\mapsto x^{d}]^{*}\mathcal{F})$ is so (see Lemma $\ref{lem : irr}$). Moreover $F-\alpha$ and $G$ are polynomials of Deligne type and
\begin{itemize}
\item[$(i)$] $V(F-\alpha X_{0}^{d})$ is a smooth projective variety for $\alpha\neq 0$.
\item[$(ii)$] $V(F-\alpha X_{0}^{d})\cap V(G)\cap V(X_{0})$ is smooth of codimension $2$ in $V(F-\alpha X_{0}^{d})$ by hypothesis. Combining this with \cite[Proposition $7.2$, Chapter $1$]{Har77} one obtains that $V(F-\alpha X_{0}^{d})\cap V(G)$ is of codimension $1$ in $V(F-\alpha X_{0}^{d})$.
\item[$(iii)$] $V(F-\alpha X_{0}^{d})\cap V(G)$ is smooth. Indeed looking at the Jacobian matrix 
\[
\begin{pmatrix}
F-\alpha X_{0}^{d}  & G \\ d\alpha X_{0}^{d-1} & 0 \\ \frac{\partial F}{\partial X_{1}} & \frac{\partial G}{\partial X_{1}}\\ \vdots & \vdots \\ \frac{\partial F}{\partial X_{n+1}} & \frac{\partial G}{\partial X_{n+1}}
\end{pmatrix}
\] 
we conclude that $P$ is a singular point if $P=[1:0:...:0]$ or $P\in V(X_{0})$. Now $[1:0:...:0]\notin V(F-\alpha X_{0}^{d})$ because $\alpha\neq 0$. Also the other case is impossible because $V(F-\alpha X_{0}^{d})\cap V(G)\cap V(X_{0})$ is smooth by $(ii)$.

\end{itemize}
Hence the sheaves $\mathcal{G}_{\alpha,e}$ are geometrically irreducible, not geometrically trivial, and they are either ramified at some $\lambda\in\overline{\mathbb{A}}_{\mathbb{F}_{q}}^{1}$ or wild ramified at $\infty$ (thanks to Remark $\ref{oss : listam}$). Then we can apply Theorem $\ref{thm : 1}$ getting
\[
\begin{split}
\sum_{\mathbf{x} : F(\mathbf{x})=\alpha}t_{\mathcal{G}_{\alpha,e}}(G(\mathbf{x}))&=q^{n-2}\sum_{b\in\mathbb{F}_{q}}t_{\mathcal{G}_{\alpha,e}}(b)+O_{d,e,n,c(\mathcal{F})}(q^{\frac{n-1}{2}})\\&=q^{n-2}\sum_{b\in\mathbb{F}_{q}^{\times}}t_{\mathcal{G}_{\alpha,e}}(b)+q^{n-2}t_{\mathcal{G}_{\alpha,e}}(0)+O_{d,e,n,c(\mathcal{F})}(q^{\frac{n-1}{2}}).
\end{split}
\]
Hence, we get
\[
\begin{split}
M&=\frac{\sqrt{q}\cdot q^{n-2}}{(d,q-1)}\sum_{\alpha\in\mathbb{F}_{q}^{\times}/\mathbb{F}_{q}^{\times d}}\sum_{b\in\mathbb{F}_{q}^{\times}}t_{\mathcal{G}_{\alpha,e}}(b)+O_{d,e,n,c(\mathcal{F})}(q^{\frac{n}{2}})\\&=\frac{ q^{n-2}}{(d,q-1)}\sum_{\alpha\in\mathbb{F}_{q}^{\times}/\mathbb{F}_{q}^{\times d}}\sum_{b\in\mathbb{F}_{q}^{\times}}\sum_{\eta\in\mathbb{F}_{q}}t_{\mathcal{F}} (\alpha\eta^{d})\psi (b\eta^{e})\\&+O_{d,e,n,c(\mathcal{F})}(q^{\frac{n}{2}})\\&=\frac{ q^{n-2}}{(d,q-1)}\sum_{\alpha\in\mathbb{F}_{q}^{\times}/\mathbb{F}_{q}^{\times d}}\sum_{b\in\mathbb{F}_{q}^{\times}}\sum_{\eta\in\mathbb{F}_{q}^{\times}}t_{\mathcal{F}} (\alpha\eta^{d})\psi (b\eta^{e})\\&+\frac{ q^{n-2}}{(d,q-1)}\sum_{\alpha\in\mathbb{F}_{q}^{\times}/\mathbb{F}_{q}^{\times d}}\sum_{b\in\mathbb{F}_{q}^{\times}}t_{\mathcal{F}} (0)+O_{d,e,n,c(\mathcal{F})}(q^{\frac{n}{2}})\\&=q^{n-2}\sum_{a\in\mathbb{F}_{q}^{\times}}\sum_{b\in\mathbb{F}_{q}^{\times}}t_{\mathcal{F}} (a)\psi (b)+(q-1)q^{n-2}t_{\mathcal{F}}(0)\\&+O_{d,e,n,c(\mathcal{F})}(q^{\frac{n}{2}}).
\end{split}
\]
For the first term of $(\ref{eq : id})$ we can argue as follows
\[
\begin{split}
\sum_{a\in\mathbb{F}_{q}}t_{\mathcal{F}}(a)|N(a,0,F,G)|&=\sum_{\mathbf{x}:G(\mathbf{x})=0}t_{\mathcal{F}} (F(\mathbf{x}))\\&= q^{n-2}\sum_{a\in\mathbb{F}_{q}}t_{\mathcal{F}}(a)+O_{d,e,n,c(\mathcal{F})}(q^{\frac{n-1}{2}}),
\end{split}
\]
again using Theorem $\ref{thm : 1}$, and similarly for the second term $\sum_{b\in\mathbb{F}_{q}}t_{\mathcal{F}}(0)\psi (b)N(0,b)$. So $(\ref{eq : id})$ becomes
\[
\begin{split}
\sum_{\mathbf{x}\in\mathbb{F}_{q}^{n}}t_{\mathcal{F}}(F(\mathbf{x}))\psi(G(\mathbf{x}))&=q^{n-2}\sum_{a\in\mathbb{F}_{q}}t_{\mathcal{F}}(a)+q^{n-2}\sum_{b\in\mathbb{F}_{q}}t_{\mathcal{F}}(0)\psi (b)\\&+ q^{n-2}\sum_{a\in\mathbb{F}_{q}^{\times}}\sum_{b\in\mathbb{F}_{q}^{\times}}t_{\mathcal{F}} (a)\psi (b)\\&+(q-1)q^{n-2}t_{\mathcal{F}}(0)+E(q)+O_{d,e,n,c(\mathcal{F})}(q^{\frac{n}{2}})
\end{split}
\]
Using the fact that the term $q^{n-2}t_{\mathcal{F}}(0)$ is counted twice in the left hand side and recalling the definition of $E(q)=-q^{n-1}t_{\mathcal{F}}(0)+O_{d,e,n,c(\mathcal{F})}(q^{\frac{n}{2}})$ we get
\[
\begin{split}
\sum_{\mathbf{x}\in\mathbb{F}_{q}^{n}}t_{\mathcal{F}}(F(\mathbf{x}))\psi(G(\mathbf{x}))&=q^{n-2}\sum_{(a,b)\in\mathbb{F}_{q}^{2}}t_{\mathcal{F}}(a)\psi(b)+O_{d,e,n,c(\mathcal{F})}(q^{\frac{n}{2}})\\&=O_{d,e,n,c(\mathcal{F})}(q^{\frac{n}{2}}),
\end{split}
\]
since $\sum_{b\in\mathbb{F}_{q}}\psi(b)=0$.

\subsubsection{Some examples.}
In the following we will denote $\mathfrak{F}$ the family of geometrically irreducible middle-extension $\ell$-adic sheaf on $\overline{\mathbb{A}}_{\mathbb{F}_{q}}^{1}$ pure of weight $0$.
\begin{lem} 
Let $\mathcal{F}\in\mathfrak{F}$ then:
\begin{itemize}
\item[$(i)$] If $e=1$ and $\mathcal{F}\in\mathfrak{F}$ is a Fourier sheaf, then $T_{1}(\mathcal{F})=\FT_{\psi}(\mathcal{F})\in\mathfrak{F}$.
\item[$(ii)$] If $e>1$ and $[\times\eta]^{*}\mathcal{F}\neq_{\geom}\mathcal{F}$ for every non trivial $e$-root of unity $\eta$, then $T_{e}(\mathcal{F})\in\mathfrak{F}$.
\item[$(iii)$] if $T_{e}(\mathcal{F})\in\mathfrak{F}$, for any $\alpha\neq 0$, $T_{e}([\times\alpha]^{*}\mathcal{F})\in\mathfrak{F}$.
\end{itemize}
\label{lem : irr}
\end{lem}
\begin{proof}[Proof of Lemma $\ref{lem : irr}$]
One argues as follows:
\begin{itemize}
\item[$(i)$] Is just an application of \cite[Theorem $8.4.1$]{Kat88}. 
\item[$(ii)$] First observe that we may assume that $\mathcal{F}$ is not geometrically trivial otherwise the result is straightforward. Let us start writing
\[
\begin{split}
\frac{1}{q}\sum_{x\in\mathbb{F}_{q}}|t_{T_{e}(\mathcal{F})}(x)|^{2}&=\frac{1}{q^{2}}\sum_{x\in\mathbb{F}_{q}}\Big|\sum_{y\in\mathbb{F}_{q}}t_{\mathcal{F}} (y)\psi(xy^{e})\Big|^{2}\\&=\frac{1}{q^{2}}\sum_{x\in\mathbb{F}_{q}}\sum_{(y_{1},y_{2})\in\mathbb{F}_{q}^{2}}t_{\mathcal{F}} (y_{1})\overline{t_{\mathcal{F}} (y_{2})}\psi(x(y_{1}^{e}-y_{2}^{e}))\\&=\frac{1}{q^{2}}\sum_{(y_{1},y_{2})\in\mathbb{F}_{q}^{2}}t_{\mathcal{F}} (y_{1})\overline{t_{\mathcal{F}} (y_{2})}\sum_{x\in\mathbb{F}_{q}}\psi(x(y_{1}^{e}-y_{2}^{e}))
\end{split}
\]
If $e>1$ and $[\times\eta]^{*}\mathcal{F}\neq_{\geom}\mathcal{F}$ for every non trivial $e$-root of unity $\eta$, we obtain 
\[
\frac{1}{q}\sum_{x\in\mathbb{F}_{q}}|t_{\mathcal{G}}(x)|^{2}=\frac{1}{q}\sum_{\eta^{e}=1}\sum_{y_{1}\in\mathbb{F}_{q}}t_{\mathcal{F}} (y_{1})\overline{t_{\mathcal{F}} (\eta y_{1})}=1+O(q^{1/2}).
\]
Applying \cite[Lemma $7.0.3$]{Kat96} we get the result.
\item[$(iii)$] Observing that 
\[
\begin{split}
t_{T_{e}([\times\alpha]^{*}\mathcal{F})}(x)&=-\frac{1}{\sqrt{q}}\sum_{z\in\mathbb{F}}\psi (z^{e}x)t_{\mathcal{F}}(\alpha z)\\&=-\frac{1}{\sqrt{q}}\sum_{w\in\mathbb{F}}\psi ((\overline{\alpha}w)^{e} x)t_{\mathcal{F}}(w)\\&= t_{T_{e}(\mathcal{F})}(\overline{\alpha}^{e}x)
\end{split}
\]
for any $x\in\mathbb{F}_{q}$, we get
\[
\frac{1}{q}\sum_{x\in\mathbb{F}_{q}}|t_{T_{e}([\times\alpha]^{*}\mathcal{F})}(x)|^{2}=\frac{1}{q}\sum_{x\in\mathbb{F}_{q}}|t_{T_{e}(\mathcal{F})}(x)|^{2}=1 +O(q^{1/2}),
\]
by the hypothesis on $T_{e}(\mathcal{F})$. Thus applying \cite[Lemma $7.0.3$]{Kat96} the result follows.
\end{itemize}
\end{proof}

Then we can prove
\begin{cor}
Let $h\in\mathbb{F}[T]$ be a polynomial and $t$ a trace function appearing in the decomposition of $1_{h(\mathbb{F}_{p})}$ in Proposition $\ref{prop : 1}$. Then
\[
\sum_{\mathbf{x}\in\mathbb{F}_{p}^{n}}t(F(\mathbf{x});p)e\Big(\frac{\langle\mathbf{x},\mathbf{u}\rangle)}{p}\Big)\ll_{d,n}p^{\frac{n}{2}},
\]
for any $F\in\mathbb{F}_{p}[X_{1},...,X_{n}]$  irreducible homogeneous polynomial of degree $d\geq 1$ such that $V(F)\subset \mathbb{P}_{\mathbb{F}_{p}}^{n-1}$ is smooth and for any $\mathbf{u}\in\mathbb{F}_{p}^{n}$ such that $V(\langle\mathbf{x},\mathbf{u}\rangle)$ is not tangent to $V(F)$. 
\label{cor : sieve}
\end{cor}
\begin{proof}
Let us denote $\mathcal{F}$ the $\ell$-adic sheaf attached to $t$. If $\mathcal{F}=_{\geom}\mathcal{K}_{\chi (T)}$ a Kummer sheaf attached to a character $\chi$ of order dividing $d$ then the result is a special case of \cite[Theorem $1$]{Kat07}. So we may assume $\mathcal{F}\neq_{\geom}\mathcal{K}_{\chi (T)}$ for any Kummer sheaf attached to a character $\chi$ of order dividing $d$. We only need to prove that $[x\mapsto x^{d}]^{*}\mathcal{F}$ is not geometrically trivial. We start writing
\[
\begin{split}
\sum_{x\in\mathbb{F}_{p}}t_{[x\mapsto x^{d}]^{*}\mathcal{F}}(x)&=\sum_{x\in\mathbb{F}_{p}}t(x^{d})\\&=\sum_{z\in\mathbb{F}_{p}}t(z)\sum_{\chi :\chi^{d}=1}\chi (z)\\&=\sum_{\chi :\chi^{d}=1}\sum_{z\in\mathbb{F}_{p}}t(z)\chi (z)\\&\ll_{c(\mathcal{F}),d}\sqrt{p},
\end{split}
\]
because $\mathcal{F}\neq_{\geom}\mathcal{K}_{\chi}$ for any character $\chi$ of order dividing $d$. Thus $[x\mapsto x^{d}]^{*}\mathcal{F}$ is not geometrically trivial. So we can apply Theorem $\ref{cor : mix}$.
\end{proof}

\begin{cor}
Let $p$ be a prime number and $m\geq 2$, and $F\in\mathbb{F}_{p}[X_{1},...,X_{n}]$ an irreducible homogeneous polynomial of degree $d>1$ such that the projective hypersurface $V(F)\subset \mathbb{P}_{\mathbb{F}_{p}}^{n-1}$ is smooth. For any $\mathbf{u}\in\mathbb{F}_{p}^{n}$ such that $V(\langle\mathbf{x},\mathbf{u}\rangle)$ is not tangent to $V(F)$ (i.e. with $V(F)\cap V(\langle\mathbf{x},\mathbf{u}\rangle)$ is smooth of codimention $2$ in $\mathbb{P}_{\mathbb{F}_{p}}^{n-1}$) one has
\[
\sum_{\mathbf{x}\in\mathbb{F}_{p}^{n}}\Kl_{m}(F(\mathbf{x});p)e\Big(\frac{\langle\mathbf{x},\mathbf{u}\rangle}{p}\Big)\ll_{d,n}p^{\frac{n}{2}},
\]
\end{cor}

\begin{proof}
Let start proving that $[x\mapsto x^{d}]^{*}\mathcal{K}\ell_{m}$ is irreducible. Thanks to \cite[Lemma $7.0.3$]{Kat96} it is enough to show that
\[
\frac{1}{q}\sum_{x\in\mathbb{F}_{q}}|\Kl_{m}(x^{d})|^{2}=1+O(p^{-1/2})
\]
Using the same argument as in Corollary $\ref{cor : sieve}$, one gets
\[
\begin{split}
\frac{1}{q}\sum_{x\in\mathbb{F}_{q}}|\Kl_{m}(x^{d})|^{2}&=\sum_{\chi:\chi^{d}=1}\sum_{z\in\mathbb{F}_{q}}\Kl_{m}(z)^{2}\chi(z)\\&=1+\sum_{\substack{\chi\neq 1:\\ \chi^{d}=1}}\sum_{z\in\mathbb{F}_{q}}\Kl_{m}(z)^{2}\chi(z)+O(p^{-1/2}).
\end{split}
\]
On the other hand, one has that
\[
\frac{1}{q}\sum_{z\in\mathbb{F}_{q}}\Kl_{m}(z)^{2}\chi(z)=1+O(p^{-1/2})
\]
if and only if $\mathcal{K}\ell_{m}\otimes\mathcal{K}\ell_{m}=_{\geom}\mathcal{K}_{\chi (T)}$ but this is not the case since $\mathcal{K}\ell_{m}\otimes\mathcal{K}\ell_{m}$ is wildly ramified at $\infty$ (\cite[Proposition $10.4.1$]{Kat88}) while $\mathcal{K}_{\chi(T)}$ is tame everywhere. Thus $[x\mapsto x^{d}]^{*}\mathcal{K}\ell_{m}$ is irreducible. Moreover $\mathcal{K}\ell_{m}$ is a Fourier sheaf, thus $T_{1}([x\mapsto x^{d}]^{*}\mathcal{K}\ell_{m})$ is irreducible (Lemma $\ref{lem : irr}$ part $(i)$). Now we have to distinguish three cases:
\begin{itemize}
\item[$(i)$] $d< m$, in this case $[x\mapsto x^{d}]^{*}\mathcal{K}\ell_{m}(\infty)$ has only one break at $\frac{d}{n}< 1$, thus $T_{1}([x\mapsto x^{d}]^{*}\mathcal{K}\ell_{m})$ is ramified at $0$ (\cite[Theorem $7.5.4$]{Kat90}),
\item[$(ii)$] $d=m$, in this case $[x\mapsto x^{d}]^{*}\mathcal{K}\ell_{m}(\infty)$ has only one break at $1$, then $T_{1}([x\mapsto x^{d}]^{*}\mathcal{K}\ell_{m})$ is singular at some $\lambda\in\overline{\mathbb{A}}_{\mathbb{F}_{q}}^{1}$ (\cite[Corollary $8.5.8$]{Kat88}),
\item[$(iii)$] $d> m$, in this case $[x\mapsto x^{d}]^{*}\mathcal{K}\ell_{m}(\infty)$ has only one break at $\frac{d}{n}> 1$, thus $T_{1}([x\mapsto x^{d}]^{*}\mathcal{K}\ell_{m})$ is wildly ramified at $\infty$ (\cite[Theorem $7.5.4$]{Kat90}).
\end{itemize}
In any case $T_{1}([x\mapsto x^{d}]^{*}\mathcal{K}\ell_{m})$ is not geometrically trivial. Thus, we can apply Theorem $\ref{cor : mix}$ and we get the result.
\end{proof}
%Let us start proving that $T_{e}([x\mapsto x^{d}]^{*}\mathcal{K}\ell_{m})$ is irreducible. Thanks to Lemma $\ref{lem : irr}$ ,it is enough to show that $\times[\eta]^{*}\mathcal{K}\ell_{m}\neq_{\geom} \mathcal{K}\ell_{m}$ for any $e$-root of unity.

%Since we have that $\mathcal{K}\ell_{m}\neq_{\geom}[\eta ]^{*}\mathcal{K}\ell_{m}$ for any $\eta$ non trivial $e$-root of unity and any $m$ even (\cite[Proposition $3.7$]{FKM1}), we have that $T_{e}(\mathcal{K}\ell_{m})$ is irreducible thanks to Lemma $\ref{lem : irr}$. Moreover $[x\mapsto x^{e}]_{*}\mathcal{K}\ell_{m}$ is wild ramified at $\infty$ with an unique break at $\frac{1}{n}$. Thus $T_{e}([x\mapsto x^{e}]_{*}\mathcal{K}\ell_{m})=\FT([x\mapsto x^{e}]_{*}\mathcal{K}\ell_{m})$ is ramified at $0$. Thus we apply Theorem $\ref{cor : mix}$ and we get the result.

\section*{Proof of Theorem $\ref{thm : box}$}
\label{sec : proof}
The strategy of the proof of Theorem $\ref{thm : box}$ is similar to the one presented by Munshi in \cite{Mun09} with some modification. Let $W:\mathbb{R}^{n+1}\longrightarrow\mathbb{R}$ be a non-negative $C^{\infty}$ function with support in the box $[-B,B]^{n+1}$ satisfying
\[
\Big|\frac{\partial^{i_{0}+...+i_{n}}W(X_{0},...,X_{n})}{\partial X^{i_{0}}...\partial X^{i_{n}}}\Big|\ll B^{-(i_{0}+...+i_{n})},
\]
for all $i_{0},...,i_{n}\geq 0$. This property leads to the following bound for the Fourier transform
\[
\hat{W}(\mathbf{u})\ll B^{n+1}\prod_{i=0}^{n}(1+|u_{i}|B)^{-\kappa}
\]
for any $\kappa>0$ and where $\mathbf{u}=(u_{0},...,u_{n})$. Let $\mathcal{P}\subset\mathcal{P}_{f}$ be a finite subset of $\mathcal{P}_{f}$ to be chosen later and define
\[
a(k):=
\begin{cases}
0 & \text{ if $k\in S$},\\
\sum_{\mathbf{x}\in\mathbb{Z}^{n+1}, F(\mathbf{x})=k}W(\mathbf{x}) & \text{ otherwise},
\end{cases}
\]
where $S:=\{k: |\{p\in\mathcal{P}:k\in S_{f,p}\mod p)\}|\geq \frac{P}{2d}\}$ and $P=|\mathcal{P}|$ . By definition $(a(k))_{k\in\mathbb{N}}$ satisfies the hypothesis of polynomial sieve, then we get
\[
\mathcal{V}_{f}(\mathcal{A})\ll_{d} P^{-1}\sum_{k}a(k)+
P^{-2}\sum_{p\neq q\in\mathcal{P}}\sum_{i,j\neq 1}\Big|\sum_{k}a(k)t_{i,p}(k)\overline{t}_{j,q}(k)\Big|.
\]
We estimate the first term in the right hand side trivially as $\sum_{k}a(k)\ll B^{n+1}$. Thus we have to bound
\[
\sum_{k}a(k)t_{i,p}(k)\overline{t}_{j,q}(k)=S(p,q,i,j)-E(p,q,i,j)
\]
for any $p\neq q$, where
\[
S(p,q,i,j):=\sum_{\mathbf{x}\in\mathbb{Z}^{n+1}}W(\mathbf{x})t_{i,p}(F(\mathbf{x}))\overline{t}_{j,q}(F(\mathbf{x}))
\]
and 
\[
E(p,q,i,j):=\sum_{\substack{k\in S\\ \mathbf{x}\in [-B,B]^{n+1} :F(\mathbf{x})=k}}W(\mathbf{x})t_{i,p}(k)\overline{t}_{j,q}(k).
\]
Using the Poisson summation formula we can rewrite $S(p,q,i,j)$ as
\[
S(p,q,i,j)=(pq)^{-(n+1)}\sum_{\mathbf{u}\in\mathbb{Z}^{n+1}}g(\mathbf{u},p,q,i,j)\hat{W}\Big(\frac{\mathbf{u}}{pq}\Big),
\] 
where
\[
g(\mathbf{u},p,q,i,j):=\sum_{\mathbf{a} \mod pq}t_{i,p}(F(\mathbf{a}))\overline{t}_{j,q}(F(\mathbf{a}))e\Big(\frac{\langle\mathbf{a},\mathbf{u}\rangle}{pq}\Big).
\]
Using the Chinese Remainder Theorem we can split the above sum as
\[
g(\mathbf{u},p,q,i,j)=g(\overline{q}\mathbf{u},t_{i,p})g(\overline{p}\mathbf{u},\overline{t}_{j,q}),
\]
with
\begin{equation}
g(\mathbf{u},t_{i,p}):=\sum_{\mathbf{a} \in\mathbb{F}_{p}^{n+1}}t_{i,p}(F(\mathbf{a}))e\Big(\frac{\langle\mathbf{a},\mathbf{u}\rangle}{p}\Big).
\label{eq : esti}
\end{equation}
The problem now is to give a good bound for $g(\mathbf{u},t_{i,p})$. Assume that the smooth variety $V(F)$ is still smooth modulo $p$ and denote by $V_{p}(F)$ its reduction modulo $p$. Moreover we denote by $V(F)^{*}$ its dual variety. Recall that the dual variety of a hypersurface is still a hypersurface \cite{ED16}[Proposition $2.9$]; we denote by $G$ the homogeneous polynomial of degree $e'$ such that $V(G)=V(F)^{*}\subset\overline{\mathbb{P}}^{n}$. We can distinguish three situations:
\begin{itemize}
\item[$i)$] $\mathbf{u}\equiv \mathbf{0}\mod p$. In this case we say that $\mathbf{u}$ is of \textit{$0$-type},
\item[$ii)$] $\mathbf{u}$ is non-zero modulo $p$ and the associate hyperplane $\langle\mathbf{a},\mathbf{u}\rangle =0$ is not tangent to $V_{p}(F)$. In this case we say that $\mathbf{u}$ is \textit{good},
\item[$iii)$] $\mathbf{u}$ is non-zero modulo $p$ and the associate hyperplane $\langle\mathbf{a},\mathbf{u}\rangle =0$ is tangent to $V_{p}(F)$. In this case we say that $\mathbf{u}$ is \textit{bad}.
\end{itemize}
Let us discuss cases $(i)$ and $(ii)$. Recall that the $t_{i}$s are trace functions attached to middle-extension sheaves of weight $0$ which are tame and geometrically irreducible. Moreover $F$ is a polynomial of Deligne type and if $u$ is good we have that $V_{p}(F)\cap V_{p}(\langle\mathbf{a},\mathbf{u}\rangle)$ is smooth of codimension $2$ in $\overline{\mathbb{P}}_{\mathbb{F}_{p}}^{n}$. Then we have
\[
\sum_{\mathbf{a} \in\mathbb{F}_{p}^{n+1}}t_{i,p}(F(\mathbf{a}))e\Big(\frac{\langle\mathbf{a},\mathbf{u}\rangle}{p}\Big)= \delta_{\mathbf{u}=\mathbf{0},p}(\mathbf{u})p^{n}\sum_{a\in\mathbb{F}_{p}}t_{i,p}(a)+O_{d,e,n,c(\mathcal{F}_{i})}(p^{\frac{n+1}{2}}),
\]
where $\delta_{\mathbf{u}=\mathbf{0},p}(\mathbf{u})=1$ if $\mathbf{u}=\mathbf{0}\mod p$ and $0$ otherwise. Observing that $c(\mathcal{F}_{i})\ll_{d} 1$ we get
\[
g(\mathbf{0},t_{i,p})\ll_{d,e,n} p^{n+\frac{1}{2}},
\]
and for $\mathbf{u}$ good
\[
g(\mathbf{u},t_{i,p})\ll_{d,e,n} p^{\frac{n+1}{2}}.
\]
For $\mathbf{u}$ bad, we have instead the following
\begin{lem}
If $\mathbf{u}$ is bad
\[
g(\mathbf{u},t_{i,p})\ll_{d,e,n}\sqrt{p}p^{\frac{n+1}{2}}.
\]
\label{lem : bad}
\end{lem}
\begin{proof} 
We denote be $K_{i,p}$ the normalized Fourier transform of $t_{i,p}$ (which is as well a trace function because $\mathcal{F}_{i}$ is tame and then $\mathcal{F}_{i}$ is a Fourier sheaf)
\[
K_{i,p}(y)=-\frac{1}{\sqrt{p}}\sum_{b\in\mathbb{F}_{p}} t_{i,p}(b)e\Big(\frac{by}{p}\Big).
\]
Starting from the Fourier inversion formula of $t_{i,p}$ one has
\[
t_{i,p}(F(\mathbf{a}))=-\frac{1}{\sqrt{p}}\sum_{b\in\mathbb{F}_{p}}K_{i,p}(b)e\Big(\frac{bF(\mathbf{a})}{p}\Big),
\]
and this leads to
\[
\begin{split}
g(\mathbf{u},t_{i,p})&=\sum_{\mathbf{a}\in\mathbb{F}_{p}^{n+1}}\Big(-\frac{1}{\sqrt{p}}\sum_{b\in\mathbb{F}_{p}}K_{i,p}(b)e\Big(\frac{bF(\mathbf{a})}{p}\Big)\Big)e\Big(\frac{\langle\mathbf{a},\mathbf{u}\rangle}{p}\Big)\\&=-\frac{1}{\sqrt{p}}\sum_{b\in\mathbb{F}_{p}}K_{i,p}(b)\sum_{\mathbf{a}\in\mathbb{F}_{p}^{n+1}}e\Big(\frac{bF(\mathbf{a})+\langle\mathbf{a},\mathbf{u}\rangle}{p}\Big)\\&\ll_{d,e,n}\sqrt{p}p^{\frac{n+1}{2}}
\end{split}
\]
where in the final bound we used the fact that $\|K_{i,p}\|_{\infty}\ll c(\mathcal{F}_{i})^{2}\ll_{d}1$ (\cite[Page $7$ combined with Paragraph $3.4$]{FKM14b}) and the Deligne's bound for additive character sums (\cite[Theorem $8.4$]{Del74}).
\end{proof}

At this point we choose the set $\mathcal{P}$ as
\[
\mathcal{P}:=\{B^{\frac{n+1}{n+2}}(\log B)^{\frac{1}{n+2}}\leq p\leq 2B^{\frac{n+1}{n+2}}(\log B)^{\frac{1}{n+2}}:f(\mathbb{F}_{p})\neq \mathbb{F}_{p},\text{ $p$ is of good reduction for } V(F)\}.
\]
Notice that since $F\in\mathbb{Z}[X_{0},...,X_{n}]$ defines a smooth projective hypersurface $V(F)$, $V_{p}(F)$ is smooth for all but finitelly many $p$. To see this we consider the resultant $r:=\text{Res}(F,\frac{\partial F}{\partial X_{0}},...,\frac{\partial F}{\partial X_{n}})$. Since $F,\frac{\partial F}{\partial X_{0}},...,\frac{\partial F}{\partial X_{n}}$ do not have any non trivial common solutions ($V(F)$ is smooth), $r\in\mathbb{Z}\setminus\{0\}$. On the other hand, $V_{p}(F)$ is singular if and only if $p|r$ (\cite{Char92}[Section IV]). Thus $V_{p}(F)$ is singular for $\omega (r)$ prime numbers. Then, $P=|\mathcal{P}|\sim \Big(\frac{B}{\log B}\Big)^{\frac{n+1}{n+2}}$ (the set of prime $p$ such that $f(\mathbb{F}_{p})\neq\mathbb{F}_{p}$ has positive density in the set of the prime numbers). Thus, we obtain
\[
\mathcal{V}_{f}(\mathcal{A})\ll_{d,e,n}B^{n+\frac{1}{n+2}}(\log B)^{\frac{n+1}{n+2}} +\frac{1}{P^{2}}\sum_{p\neq q\in\mathcal{P}}S(p,q)+ \frac{1}{P^{2}}\sum_{p\neq q\in\mathcal{P}}E(p,q).
\]
where $S(p,q):=\sum_{i,j\neq 1}|S(p,q,i,j)|$ and $E(p,q):=\sum_{i,j\neq 1}|E(p,q,i,j)|$. To conclude now it is enough to analyze the contribution of the $S(p,q)$s and of the $E(p,q)$s. We start studying the contribution of the $S(p,q)$s:
\begin{lem}
We have that
\[
\frac{1}{P^{2}}\sum_{p\neq q\in\mathcal{P}}S(p,q)\ll_{d,e,n}B^{n+\frac{1}{n+2}}(\log B)^{\frac{n+1}{n+2}}.
\]
\label{lem : avesieve}
\end{lem}
\begin{proof}
To simplify the notation, in the following we denote $Q:=B^{\frac{n+1}{n+2}}(\log B)^{\frac{1}{n+2}}$: then we have that for any $p,q\in\mathcal{P}$, $Q\leq p,q\leq 2Q$. Moreover, let us denote $S_{g,g}(p,q)$ (resp. $S_{0,0}(p,q)$, $S_{0,g}(p,q)$, $S_{g,0}(p,q)$, $S_{b,b}(p,q)$, $S_{b,g}(p,q)$, $S_{g,b}(p,q)$, $S_{b,0}(p,q)$, $S_{0,b}(p,q)$), the contribution to $S(p,q)$ of the $\mathbf{u}$'s which are good for both $p$ and $q$ (resp. the contribution to $S(p,q)$ of the $\mathbf{u}$'s which are of type $0$ for both $p$ and $q$, and so on). Let us start with the contribution of the $\mathbf{u}\in\mathbb{Z}^{n+1}$ which are good for both $p$ and $q$:
\[
S_{g,g}(p,q)\ll B^{n+1}\frac{(pq)^{\frac{n+1}{2}}}{(pq)^{n+1}}\sum_{\mathbf{u}\in\mathbb{Z}^{n+1}}\prod_{i=0}^{n}\Big(1+\frac{|u_{i}|B}{pq}\Big)^{-2}\ll (pq)^{\frac{n+1}{2}}
\]
where the last step follows from the fact that $pq\geq B$. Thus we conclude that
\[
\frac{1}{P^{2}}\sum_{p\neq q\in\mathcal{P}}S_{g,g}(p,q)\ll Q^{n+1}=B^{n+\frac{1}{n+2}}(\log B)^{\frac{n+1}{n+2}}.
\]
Let us now discuss the contribution of $S_{0,0}(p,q)$s
\[
\begin{split}
S_{0,0}(p,q)&\ll \frac{B^{n+1}\cdot (pq)^{n+\frac{1}{2}}}{(pq)^{n+1}}\sum_{\substack{\mathbf{u}\in\mathbb{Z}^{n+1},\\ \mathbf{u}\equiv 0\text{ } (pq)}}\prod_{i=0}^{n}\Big(1+\frac{|u_{i}|B}{pq}\Big)^{-2}\\&\ll \frac{B^{n+1}}{(pq)^{\frac{1}{2}}}.
\end{split}
\]
Thus
\[
\frac{1}{P^{2}}\sum_{p\neq q\in\mathcal{P}}S_{0,0}(p,q)\ll \frac{B^{n+1}}{Q}\ll B^{n+\frac{1}{n+2}}
\]
since $Q=B^{\frac{n+1}{n+2}}(\log B)^{\frac{1}{n+2}}$. Let us continue with the contribution of the $S_{b,b}(p,q)$, i.e.
\[
\frac{1}{P^{2}}\sum_{p\neq q\in\mathcal{P}}S_{b,b}(p,q).
\]
We have that, if $\mathbf{u}$ is bad for $p$ and $q$ then 
\[
S_{b,b}(p,q)\ll\frac{B^{n+1}\cdot (pq)^{\frac{n+2}{2}}}{(pq)^{n+1}}\sum_{\substack{\mathbf{u}\in\mathbb{Z}^{n+1}\\ \mathbf{u}\text{ is bad for }p\\ \mathbf{u}\text{ is bad for }q}}\prod_{i=0}^{n}\Big(1+\frac{|u_{i}|B}{pq}\Big)^{-\kappa}
\]
thanks to Lemma $\ref{lem : bad}$. Thus
\[
\frac{1}{P^{2}}\sum_{\substack{p,q\in\mathcal{P}\\p\neq q}}S_{b,b}(p,q)\ll \frac{Q^{n+2}\cdot B^{n+1}}{P^{2}\cdot Q^{2(n+1)}}\sum_{\substack{p,q\in\mathcal{P}\\p\neq q}}\sum_{\substack{\mathbf{u}\in\mathbb{Z}^{n+1}\\ \mathbf{u}\text{ is bad for }p\\ \mathbf{u}\text{ is bad for }q}}\prod_{i=0}^{n}\Big(1+\frac{|u_{i}|B}{pq}\Big)^{-\kappa}.
\]
Now, we recall that $G$ is the polynomial associated to the dual variety $V(F)^{*}\subset\overline{\mathbb{P}}^{n}$ and that $\deg G=e'$. We rewrite the above sum as
\begin{equation}
\frac{B^{n+1}}{P^{2}Q^{n}}\sum_{\substack{p,q\in\mathcal{P}\\p\neq q}}\sum_{\substack{\mathbf{u}\in\mathbb{Z}^{n+1}\\ \mathbf{u}\text{ is bad for }p\\ \mathbf{u}\text{ is bad for }q}}\prod_{i=0}^{n}\Big(1+\frac{|u_{i}|B}{pq}\Big)^{-\kappa}=\frac{B^{n+1}}{P^{2}Q^{n}}\sum_{\mathbf{u}\in\mathbb{Z}^{n+1}}\sum_{\substack{p,q\in\mathcal{P}\\p\neq q\\\\ \mathbf{u}\text{ is bad for }p\\ \mathbf{u}\text{ is bad for }q}}\prod_{i=0}^{n}\Big(1+\frac{|u_{i}|B}{pq}\Big)^{-\kappa}.
\label{eq : Type II}
\end{equation}
Now we split the sum over $\mathbf{u}$ in the right hand side of $(\ref{eq : Type II})$ in two
\[
\sum_{\mathbf{u}\in\mathbb{Z}^{n+1}}=\sum_{\substack{\mathbf{u}\in\mathbb{Z}^{n+1}\\G(\mathbf{u})\neq 0}}+\sum_{\substack{\mathbf{u}\in\mathbb{Z}^{n+1}\\G(\mathbf{u})= 0}}.
\]
Let us bound the contribution of the $\mathbf{u}$s which satisfy $G(\mathbf{u})\neq 0$. We need to estimate the numbers of $p$ prime numbers such that $\mathbf{u}$ is bad modulo $p$. Since $G(\mathbf{u})\neq 0$, the variety $V(F)\cap V(\langle \mathbf{X},\mathbf{u}\rangle)$\footnote{We denote by $\langle \mathbf{X},\mathbf{u}\rangle$ the linear equation $ u_{0}X_{0}+...+u_{n}X_{n}$} is smooth, i.e. the Jacobian
\[
\begin{pmatrix}
F & \langle \mathbf{X},\mathbf{u}\rangle\\ \frac{\partial F}{\partial{X_{0}}} & u_{0}\\ \vdots & \vdots\\ \frac{\partial F}{\partial{X_{n}}} & u_{n}
\end{pmatrix}
\]
has maximal rank for any $x\in V(F)\cap V(\langle \mathbf{X},\mathbf{u}\rangle)$. Then if we define $H_{0,\mathbf{u}}=F$, $H_{n+1,\mathbf{u}}=\langle \mathbf{X},\mathbf{u}\rangle$ and
\[
H_{i,\mathbf{u}}:=\det
\begin{pmatrix}
\frac{\partial F}{\partial{X_{i-1}}} & u_{i-1}\\\frac{\partial F}{\partial{X_{i}}} & u_{i}
\end{pmatrix}
\]
for any $i=1,...,n$, we get that the resultant of the $H_{i,\mathbf{u}}$s, $\text{Res}(H_{0,\mathbf{u}},...,H_{n+1,\mathbf{u}})$, is a non zero integer. In the following for any polynomial $H\in\mathbb{Z}[X_{0},...,X_{n}]$ we denote by $\overline{H}:= H\mod p$. If $\mathbf{u}$ is bad modulo $p$, then $V_{p}(F)\cap V_{p}(\langle \mathbf{X},\mathbf{u}\rangle)=V(\overline{F}), \cap V(\overline{\langle \mathbf{X},\mathbf{u}\rangle)})$ is singular and this is true if and only if $\overline{H}_{0,\mathbf{u}},...,\overline{H}_{n+1,\mathbf{u}}$ have a common non trivial root. On the other hand, $\overline{H}_{0,\mathbf{u}},...,\overline{H}_{n+1,\mathbf{u}}$ have a common root if and only if $p|\text{Res}(H_{0,\mathbf{u}},...,H_{n+1,\mathbf{u}})$ (\cite{Char92}[Section IV]). Then we conclude that 
\[
|\{p:\mathbf{u}\text{ is bad}\mod p\}|=\omega (\text{Res}(H_{0,\mathbf{u}},...,H_{n+1,\mathbf{u}})),
\]
where for any $n$, $\omega (n)$ denotes the number of distinct prime factors of $n$.  Using \cite{GKZ08}[Proposition $1.1$, Chapter $13$], we conclude that $\omega (\text{Res}(H_{0,\mathbf{u}},...,H_{n+1,\mathbf{u}}))\ll_{d,\|F\|}\log \|\mathbf{u}\|$
Thus
\[
\begin{split}
\frac{B^{n+1}}{P^{2}Q^{n}}\sum_{\substack{\mathbf{u}\in\mathbb{Z}^{n+1}\\G(\mathbf{u})\neq 0}}\sum_{\substack{p,q\in\mathcal{P}\\p\neq q\\\\ \mathbf{u}\text{ is bad for }p\\ \mathbf{u}\text{ is bad for }q}}\prod_{i=0}^{n}\Big(1+\frac{|u_{i}|B}{pq}\Big)^{-\kappa}&\ll \frac{B^{n+1}}{P^{2}Q^{n}}\sum_{\substack{\mathbf{u}\in\mathbb{Z}^{n+1}\\G(\mathbf{u})\neq 0}}\sum_{\substack{p,q\in\mathcal{P}\\p\neq q\\\\ \mathbf{u}\text{ is bad for }p\\ \mathbf{u}\text{ is bad for }q}}\prod_{i=0}^{n}\Big(1+\frac{|u_{i}|B}{4Q^{2}}\Big)^{-\kappa}\\&\ll\frac{B^{n+1}}{P^{2}Q^{n}}\sum_{\substack{\mathbf{u}\in\mathbb{Z}^{n+1}\\ \mathbf{u}\neq \mathbf{0}}}\Big(\prod_{i=0}^{n}\Big(1+\frac{|u_{i}|B}{4Q^{2}}\Big)^{-\kappa}\Big)(\log \|\mathbf{u}\|)^{2}\\&\ll\frac{B^{n+1}}{P^{2}Q^{n}}\cdot\frac{Q^{2(n+1)+\varepsilon}}{B^{n+1+\varepsilon}}\\&\ll Q^{n+\varepsilon},
\end{split}
\]
for some $\varepsilon> 0$. Let us bound the contribution of the $\mathbf{u}$s which satisfy $G(\mathbf{u})=0$. We start writing
\[
\begin{split}
\frac{B^{n+1}}{P^{2}Q^{n}}\sum_{\substack{\mathbf{u}\in\mathbb{Z}^{n+1}\\G(\mathbf{u})= 0}}\sum_{\substack{p,q\in\mathcal{P}\\p\neq q}}\prod_{i=0}^{n}\Big(1+\frac{|u_{i}|B}{pq}\Big)^{-\kappa}&\leq\frac{B^{n+1}}{P^{2}Q^{n}}\sum_{\substack{\mathbf{u}\in\mathbb{Z}^{n+1}\\G(\mathbf{u})= 0}}\sum_{\substack{p,q\in\mathcal{P}\\p\neq q}}\prod_{i=0}^{n}\Big(1+\frac{|u_{i}|B}{4Q^{2}}\Big)^{-\kappa}\\&\leq\frac{B^{n+1}}{Q^{n}}\sum_{\substack{\mathbf{u}\in\mathbb{Z}^{n+1}\\G(\mathbf{u})= 0}}\prod_{i=0}^{n}\Big(1+\frac{|u_{i}|B}{4Q^{2}}\Big)^{-\kappa}.
\end{split}
\]
Let $\alpha >0$ be a parameter to be chosen later. We denote by
\[
\mathcal{C}_{\alpha}:=\Big[-\frac{Q^{2+2\alpha}}{B^{1+\alpha}},\frac{Q^{2+2\alpha}}{B^{1+\alpha}}\Big]\times\cdots\times\Big[-\frac{Q^{2+2\alpha}}{B^{1+\alpha}},\frac{Q^{2+2\alpha}}{B^{1+\alpha}}\Big].
\]
Then
\begin{equation}
\begin{split}
\frac{B^{n+1}}{Q^{n}}\sum_{\substack{\mathbf{u}\in\mathbb{Z}^{n+1}\\G(\mathbf{u})= 0}}\prod_{i=0}^{n}\Big(1+\frac{|u_{i}|B}{4Q^{2}}\Big)^{-\kappa}=\frac{B^{n+1}}{Q^{n}}\Big(\sum_{\substack{\mathbf{u}\in\mathcal{C}_{\alpha}\\G(\mathbf{u})= 0}}\prod_{i=0}^{n}\Big(1+\frac{|u_{i}|B}{4Q^{2}}\Big)^{-\kappa}+\sum_{\substack{\mathbf{u}\in\mathbb{Z}^{n+1}\setminus\mathcal{C}_{\alpha}\\G(\mathbf{u})= 0}}\prod_{i=0}^{n}\Big(1+\frac{|u_{i}|B}{4Q^{2}}\Big)^{-\kappa}\Big)
\end{split}
\label{eq : avabad}
\end{equation}
Let us start with the first term in the right hand side of $(\ref{eq : avabad})$. First, we have
\[
\frac{B^{n+1}}{Q^{n}}\sum_{\substack{\mathbf{u}\in\mathcal{C}_{\alpha}\\G(\mathbf{u})= 0}}\prod_{i=0}^{n}\Big(1+\frac{|u_{i}|B}{4Q^{2}}\Big)^{-\kappa}\ll\frac{B^{n+1}}{Q^{n}}\sum_{\substack{\mathbf{u}\in\mathcal{C}_{\alpha}\\G(\mathbf{u})= 0}}1.
\]
We distinguish two cases:
\begin{itemize}
\item[$i)$] $(e',n)\neq (2,2)$. We have that
\[
\sum_{\substack{\mathbf{u}\in\mathcal{C}_{\alpha}\\G(\mathbf{u})= 0}}1\leq \Big(\frac{Q^{2+2\alpha}}{B^{1+\alpha}}\Big)^{n-1+\frac{1}{e'}+\varepsilon},
\]
for any $\varepsilon > 0$ (\cite{Pila95}). Thus
\[
\frac{B^{n+1}}{Q^{n}}\sum_{\substack{\mathbf{u}\in\mathcal{C}_{\alpha}\\G(\mathbf{u})= 0}}\prod_{i=0}^{n}\Big(1+\frac{|u_{i}|B}{4Q^{2}}\Big)^{-\kappa}\ll B^{2-g(\alpha)}Q^{n-2+2g(\alpha)+\varepsilon '},
\]
where $g(\alpha):=1/e'+\alpha (n-1+ 1/e')$. The other term in the right hand side of $(\ref{eq : avabad})$ is bounded by
\[
\begin{split}
\frac{B^{n+1}}{Q^{n}}\sum_{\substack{\mathbf{u}\in\mathbb{Z}^{n+1}\setminus\mathcal{C}_{\alpha}\\G(\mathbf{u})= 0}}\prod_{i=0}^{n}\Big(1+\frac{|u_{i}|B}{4Q^{2}}\Big)^{-\kappa}&\ll \frac{B^{n+1}}{Q^{n}}\sum_{i=0}^{n}\sum_{\substack{\mathbf{u}\in\mathbb{Z}^{n+1}\\ |u_{i}|\geq \frac{Q^{2+2\alpha}}{B^{1+\alpha}} \\G(\mathbf{u})= 0}}\prod_{i=0}^{n}\Big(1+\frac{|u_{i}|B}{4Q^{2}}\Big)^{-\kappa}\\&\ll\frac{B^{n+1}}{Q^{n}}\cdot \Big(\frac{B^{\alpha}}{Q^{2\alpha}}\Big)^{\kappa -1}\cdot \Big(\frac{Q^{2}}{B}\Big)^{n}\\&\ll Q^{n-2\alpha(\kappa -1)}B^{1+\alpha(\kappa -1)}.
\end{split}
\]
Thus, we conclude that
\[
P^{-2}\sum_{\substack{p,q\in\mathcal{P}\\p\neq q}}S_{b,b}(p,q)\ll Q^{n-2\alpha(\kappa -1)}B^{1+\alpha(\kappa -1)}+B^{2-g(\alpha)}Q^{n-2+2g(\alpha)+\varepsilon}+ Q^{n+\varepsilon},
\]
for any $\varepsilon >0$. Our goal is two show that
\[
B^{2-g(\alpha)}Q^{n-2+2g(\alpha)+\varepsilon},Q^{n-2\alpha(\kappa -1)}B^{1+\alpha(\kappa -1)}\ll Q^{n+1}.
\]
To do so, we need to choose suitable values for $\alpha,\kappa$. Using the fact that $Q=B^{\frac{n+1}{n+2}}(\log B)^{\frac{1}{n+2}}=B^{1-\frac{1}{n+2}}(\log B)^{\frac{1}{n+2}}$, we get
\[
B^{2-g(\alpha)}B^{(1-1/(n+2))(n-2+2g(\alpha)+\varepsilon)}\ll B^{n+1/(n+2)}(\log B)^{\frac{n+1}{n+2}}
\]
if and only if
\[
\alpha < \Big(1-\frac{3}{n+2}-\frac{1}{e'}+\frac{2}{e'(n+2)}\Big)\beta (e',n)^{-1},
\]
where $\beta (e',n):=(n-1+1/e')(1-2/(n+2))>0$ for any $e',n\geq 1$. On the other hand choosing $\kappa\geq \alpha^{-1}+1$, we get
\[
Q^{n-2\alpha(\kappa -1)}B^{1+\alpha(\kappa -1)}\leq Q^{n-2}B^{2}=B^{n-(n-2)/(n+2)}\log B,
\]
as we wanted.
\item[$ii)$] $(e',n)=(2,2)$. In this case one use \cite{HB02}[Theorem $3$], getting
\[
\sum_{\substack{\mathbf{u}\in\mathcal{C}_{\alpha}\\G(\mathbf{u})= 0}}1\leq \Big(\frac{Q^{2+2\alpha}}{B^{1+\alpha}}\Big)^{1+\varepsilon},
\]
for any $\varepsilon > 0$ instead of \cite{Pila95}. Then the argument is the same.

\end{itemize}
The bounds for the contribution of $S_{g,0}(p,q)$, $S_{g,b}(p,q)$ and $S_{0,b}(p,q)$ are similar.
\end{proof}

Using this Lemma we obtain that
\[
\mathcal{V}_{f}(\mathcal{A})\ll_{d,e,n} B^{n+\frac{1}{n+2}}(\log B)^{\frac{n+1}{n+2}} +\frac{1}{P^{2}}\sum_{p\neq q\in\mathcal{P}}E(p,q).
\]
\begin{lem}
For any $p,q\in\mathcal{P}$, one has
\[
E(p,q,i,j)\ll_{d,e,n,\|f\|,\|F\|} B^{n}.
\]
\label{lem : rem}
\end{lem}
\begin{proof}
Thanks to \cite[Theorem 1]{HB02} one has
\[
E(p,q,i,j)=\sum_{\substack{k\in S\\ \mathbf{x}\in [-B,B]^{n+1} :F(\mathbf{x})=k}}W(\mathbf{x})t_{i,p}(k)\overline{t}_{j,q}(k)\ll_{d,e,n,\|F\|} B^{n}\sum_{k\in S, k\leq M}1,
\]
where $M:=\max_{\mathbf{x}\in [-B,B]^{n+1}}|F(\mathbf{x})|\ll_{\|F\|} B^{e}$. We recall that $S=\{k: |\{p\in\mathcal{P}:(k \mod p\in S_{f,p})\}|\geq \frac{P}{2d\log P}\}$ and that $S_{f,p}=\{f(x):x\in\overline{\mathbb{F}}_{p}, f'(x)=0\}\cap\mathbb{F}_{p}$. Moreover, we denote by $V:=\{z\in\mathbb{C}:f'(z)=0\}$. Suppose that $k\in S$ and $k\notin f(V)\cap\mathbb{Z}$. This implies that the polynomial $f(T)-k$ is separable, i.e. it has $d$ distinct roots. Thus, the discriminant of $f(T)-k$, $\text{Dicr}_{T}(f(T)-k)$, is non zero. On the other hand, if $k\in S_{f,p}$ then $f(T)-k\mod p$ has multiple roots and then $p|\text{Dicr}_{T}(f(T)-k)$. Now let us denote $m:=\prod_{p:k\in S_{f,p}}p$. Since $m|\text{Dicr}_{T}(f(T)-k)$ and $\text{Dicr}_{T}(f(T)-k)\neq 0$ we conclude that $m\leq |\text{Dicr}_{T}(f(T)-k)|$. Moreover, the discriminant of a generic polynomial can be seen as an homogeneous polynomial in the coefficients of degree $2d-2$. Hence, $\text{Dicr}_{T}(f(T)-k)|\ll_{d,\|f\|}k^{2d-2}$ and this implies that $m^{\frac{1}{2d-2}}\ll_{d,\|f\|}k$. By definition $m\geq (B^{\frac{n+1}{n+2}})^{\frac{P}{2d\log P}}$, then $k\gg (B^{\frac{n+1}{n+2}})^{\frac{P}{(4d^{2}-4d)\log P}}$. Using the fact that $P\sim\Big(\frac{B}{\log B}\Big)^{\frac{n+1}{n+2}}$ we conclude that $k\gg B^{2e}$ for $B$ large enough. Hence,
\[
\sum_{k\in S, k\leq M}1=\sum_{k\in f(V)\cap\mathbb{Z}}1\leq d-1
\]
\end{proof}
Thus 
\[
\mathcal{V}_{f}(\mathcal{A})\ll_{d,e,n,\|f\|,\|F\|} B^{n+\frac{1}{n+1}}(\log B)^{\frac{n+1}{n+2}}, 
\]
as we wanted.
\newline\newline

\textit{Acknowledgments:} I am most thankful to my advisor, Emmanuel Kowalski, for suggesting this problem and for his guidance during these years. I thank Lillian Pierce for her suggestions, comments and for her careful reading of this paper, in particular for Lemma $\ref{lem : avesieve}$. Finally, I thank Ritabrata Munshi for suggesting the averaging argument in Lemma $\ref{lem : avesieve}$.

\bibliography{bibsivdel.bib}
\end{document}